\documentclass[twoside,12pt]{amsart}
\usepackage[mathscr]{eucal}
\usepackage{amsmath,amsfonts}
\usepackage[dvips]{graphicx}
\usepackage{color}
\usepackage{amsthm}
\usepackage{amssymb}
\usepackage{latexsym, bbm, bm, subcaption, pict2e, multirow}
\allowdisplaybreaks

\def\Xint#1{\mathchoice
{\XXint\displaystyle\textstyle{#1}}%
{\XXint\textstyle\scriptstyle{#1}}%
{\XXint\scriptstyle\scriptscriptstyle{#1}}%
{\XXint\scriptscriptstyle\scriptscriptstyle{#1}}%
\!\int}
\def\XXint#1#2#3{{\setbox0=\hbox{$#1{#2#3}{\int}$ }
\vcenter{\hbox{$#2#3$ }}\kern-.58\wd0}}

\def\dashint{\Xint-}

\newtheorem{theorem}{Theorem}[section]
\newtheorem{lemma}[theorem]{Lemma}

\numberwithin{equation}{section}
\newtheorem{remark}{Remark}
\numberwithin{remark}{section}

\begin{document}
\title[Maximum-principle-satisfying discontinuous Galerkin schemes]{Third order maximum-principle-satisfying DG schemes for convection-diffusion problems with anisotropic diffusivity}
\author{
Hui Yu and Hailiang Liu
}
\address{Tsinghua University, Yau Mathematical Sciences Center, Beijing, China 100084} \email{huiyu@tsinghua.edu.cn}
\address{Iowa State University, Mathematics Department, Ames, IA 50011} \email{hliu@iastate.edu}
\date{\today}

\begin{abstract} 
For a class of convection-diffusion equations with variable diffusivity, we construct third order accurate discontinuous Galerkin (DG) schemes on both one and two dimensional rectangular meshes.
The DG method with an explicit time stepping can well be applied to nonlinear convection--diffusion equations.  
It is shown that under suitable time step restrictions, the scaling limiter proposed in [Liu and Yu, SIAM J. Sci. Comput. 36(5): A2296--A2325, 2014] when coupled with the present DG schemes preserves the solution bounds indicated by the initial data, i.e., the maximum principle, while maintaining uniform third order accuracy.
These schemes can be extended to rectangular meshes in three dimension.
The crucial for all model scenarios is that  an effective test set can be identified to verify the desired bounds of numerical solutions. 
This is achieved mainly by taking advantage of the flexible form of the diffusive flux  and the adaptable decomposition of weighted cell averages. 
Numerical results are presented to validate the numerical methods and demonstrate their effectiveness.
\end{abstract}

\maketitle

\section{Introduction} 
In this paper, we construct and analyze third order maximum-principle-satisfying (MPS) discontinuous Galerkin (DG) schemes for the following problem,
 \begin{equation}\label{md}
\left\{
  \begin{array}{ll}
  M(x) \partial_t u + \nabla\cdot f(u) = \nabla \cdot(A\nabla u), & x\in \mathbb{R}^d, \; t>0,\\
    u(0,x)=u_0 (x), & x\in \mathbb{R}^d.
  \end{array}
\right.
\end{equation}
This equation can be considered as a model for numerous physical problems. 
In this equation,  $u=u(t, x)$ is a time-dependent unknown scalar function, $f(u)$ is the smooth vector flux,  and $d$ is the spatial dimension. We assume that the diffusion tensor $A=A(x, u)$ is a symmetric and nonnegative definite matrix, and $M(x)$ is strictly positive scalar function. In this paper, we present our results and analysis for $d=1, 2$, extension to three dimension can be done as well. 

Our concern for (\ref{md}) arises from several typical scenarios. 
The first case with $f=0$ is a model for heat conduction in a non-uniform body with $M(x)=\frac{c(x)}{\rho(x)}$, where $c(x)$ is the specific heat and $\rho(x)$ the mass density. 
The heat flux is proportional to the temperature difference $-A \nabla u$, known as the Fourier law of heat conduction. 
Here $A$ measures the ability of the material to conduct heat, called the thermal conductivity. 
In the second case with  $M=1$, we have the usual convection-diffusion equation, 
\begin{equation}\label{conv_dif}
\partial_t u + \nabla \cdot f(u)  = \nabla \cdot(A(x, u) \nabla u). 
\end{equation}
There are many interpretations and derivations from fluid dynamics and other application areas that motivate the convection-diffusion equation, 
such as Navier-Stokes equations and the porous medium equation.  The third case is the Fokker-Planck equation, 
\begin{align}\label{FP}
\partial_t \rho =\nabla \cdot(\nabla \rho +\nabla V(x) \rho),
\end{align}
which can be rewritten in terms of $u=\rho e^{V(x)}$ as (\ref{md}) with $A(x) = M(x)=e^{-V(x)}$, and $f=0$. This setting allows for many variants.  
	
The above three types of problems can be formulated under the model class (\ref{md}). 
One important solution to (\ref{md}) is the one that is bounded by two constants dictated by the initial data, leading to the so-called Maximum Principle (MP). 
In other words, if
\[
c_1 = \min_{x} u_0,	\quad c_2 = \max_{x} u_0,	
\]
then $u(t,x) \in [c_1, c_2]$  for any $x\in \mathbb{R}^d$ and $t>0$.  

From the analytical viewpoint, the maximum principle is quite general but very significant due to its physical implications.  From the numerical viewpoint, it is widely recognized that maximum principle provides a valuable tool in proving solvability results (existence and uniqueness of discrete solutions), enforcing numerical stability, and deriving convergence results (a priori error estimates) for the sequence of approximate solutions.
Design of a high order scheme to preserve the maximum principle  is known a challenging task. Our goal is to better understand how a high order DG scheme can be constructed for (\ref{md})  to respect the MPS property.
The main difficulty in the anisotropic case with a variable weight $M(x)$ is the derivation of suitable sufficient conditions so that the weighted cell 
average stays in $[c_1,c_2]$ during the time evolution. Such weighted cell average is essentially used for limiting the numerical solution into $[c_1, c_2]$, 
without destroying accuracy.   

\subsection{Related work} 
Early discussion of the discrete maximum principle for the convection-diffusion equations includes the linear finite element solutions for parabolic equations \cite{Fu73},  and recent developments \cite{FHK05, FH06, TW08, VKH10, FK12}, as well as \cite{MH85} by the Petrov--Galerkin finite element method to solve convection dominated problems. However, they are under a different framework. 

The present investigation involves the choice of  numerical fluxes and monotonicity of weighted cell averages in terms of point values. In the largest sense, the origins of  these ideas  go all the way back to monotone schemes for hyperbolic conservation laws.  

Indeed, for scalar conservation laws, i.e., (\ref{conv_dif}) with $A=0$,  many first order classical schemes can be shown to be MPS (other names of this sort include bound-preserving,  positivity preserving, or maximum-principle-preserving) since such low order accurate  schemes are usually monotone. On the other hand, the Godunov Theorem states that a linear monotone scheme is at most first order accurate for the convection equation \cite{Le92}. 
To construct high order accurate MPS schemes for scalar convection, weak monotonicity in finite volume type schemes including DG methods was first used in \cite{ZS10, ZS11, ZLS12}. 
Here by weak monotonicity it means that each cell average  is monotone with respect to point values in that cell, see e.g.  \cite{Zh17}.  
The main idea in their work  is to find sufficient conditions to preserve the desired bounds of cell averages by repeated convex combinations.  A simple and efficient local MPS limiter can then be used to control the solution values at test points without affecting accuracy and conservation. Together with strong stability preserving (SSP) Runge--Kutta or multistep methods \cite{GKS11},  which are convex combinations of several formal forward Euler steps, a high order accurate finite volume or DG scheme can be rendered MPS with the limiter. 

For diffusion,  a  linear finite volume scheme can only be up to second order accurate in order to preserve the weak monotonicity, unless a non-conventional discretization is used in the scheme construction such as that in \cite{ZLS12}. 
For DG methods, in general only second order accuracy can be obtained to feature the MPS property; see \cite{ZZS13} for solving (\ref{conv_dif}) on triangular meshes. 

The only DG method known to  satisfy the weak monotonicity up to third order accuracy is the direct DG  (DDG) method introduced 
in \cite{LY09, LY10}.   
Indeed, the special method  parameters in the DDG discretization allowed us to design in \cite{LY14}  a third order MPS method for the linear Fokker-Planck equation (\ref{FP}). 
A key idea in  \cite{LY14} is the use of the non-logrithimic Landau formulation 
\[
M\partial_t u= \nabla \cdot \left(M\nabla u\right) \quad\text{ with } M(x) = e^{-V(x)} \text{ and } u=\frac{\rho}{M},
\]
so that the corresponding maximum principle on $\rho(t,x)$: 
\begin{equation*} 
c_1 e^{-V}\leq \rho(0, x) \leq c_2 e^{-V}  \Longrightarrow c_1 e^{-V}\leq \rho(t, x) \leq c_2 e^{-V} \quad \forall t > 0,
\end{equation*}
reduces to  
\[
c_1 \leq u(0, x) \leq c_2   \Longrightarrow c_1 \leq u(t, x) \leq c_2 \quad \forall t >0.
\]
With this reformulation, one can show that each weighted cell average is monotone in terms of point values under appropriate CFL conditions. The result in \cite{LY14} is directly applicable to multi-dimensional diffusion on rectangular meshes.  However, it gets subtle to ensure the MPS property on unstructured meshes; we refer to \cite{CHY16} for a third order such DDG method to solve diffusion equations on unstructured triangular meshes.

Another approach towards a positivity-preserving scheme with high order accuracy is to use the local DG (LDG) method \cite{BR97, CS98}, combined with some special positivity-preserving fluxes.
 Such an effort was first made in \cite{Zh17} for constructing high order accurate positivity-preserving DG schemes for compressible Navier--Stokes equations. Concerning further developments in this direction, we refer to \cite{SCS18, SPZ18} for solving convection--diffusion problems.   
 An MPS third-order LDG method  using overlapping meshes has been recently proposed in \cite{DY19} for convection-diffusion equations.

One noteworthy alternative to enforce positivity in high order schemes is to take a convex combination of high order flux with a first order positivity-preserving one; the method has been applied to various high order schemes including finite difference,  finite volume,  and DG methods \cite{JX13, XQX15, YXQX16}, while rigorous justification of accuracy for such methods seems difficult. 

\subsection{Present investigation} The main distinction of our present investigation from  the  above mentioned works is the use of weighted cell averages for both ensuring the weak monotonicity and applying the scaling limiter, with special attention on difficulty caused by the  anisotropic diffusivity.  
 
 The spatial discretization explored in this work is the DDG method introduced by Liu and Yan in \cite{LY09, LY10}.  Besides the usual advantages of a DG method (see e.g. \cite{HW07, Ri08, Sh09}), one main feature of the DDG method lies in numerical  flux choices for the solution gradient, which involve second order derivatives evaluated crossing cell interfaces (see (\ref{flux_DDG_1Dp}) below). With this choice, the obtained schemes are provably stable and optimally convergent as well as superconvergent for $\beta_1\not=0$ \cite{Liu15, CLZ17}. Such method has also been successfully extended to various application problems, including Fokker-Planck type equations \cite{LY14, LY15, LW16, LW17}, and the three dimensional compressible Navier--Stokes equation \cite{CYLLL16}. 

Built upon the work \cite{LY14}, we present third order DG methods for solving the initial value problem (\ref{md}), with an application to nonlinear convection-diffusion equations of form (\ref{conv_dif}). 
Let us illustrate the main ideas via a simple one-dimensional equation subject to periodic boundary condition:
$$
M(x)\partial_t u  = \partial_x(A(x)\partial_xu). 
$$  
The computational cell is denoted by $I_j = [x_{j-\frac12},x_{j+\frac12}]$,  where $x_{j+\frac12}$'s are the grid points.
Let $\tau$ and $h$ be the time and space steps for a uniform mesh and $n$ the index for the time stepping. 
The update on $\langle u_h^n \rangle_j$, the weighted cell average of the numerical approximation $u_h^n$, is given by 
\begin{equation*}
\langle u^{n+1}_h\rangle_j =  \langle u^n_h \rangle_j - \frac{\tau} {h} \left. A\widehat{\partial_x u_h^n}\right|_{\partial I_j}
\quad \text{ where } \langle u \rangle_j:= \frac{1}{h}\int_{I_j}M(x)u(x)\,dx.
\end{equation*}
Note that $\widehat{\partial_x u_h}$ is the approximation to $\partial_x u$ at $\partial I_j$, the cell interface of $I_j$, given by 
\begin{align*}
\widehat{\partial_x u}& =\frac{\beta_0}{h}[u] +\{\partial_x u\} +\beta_1 h[\partial_x^2 u].
\end{align*}
As mentioned above, this form of the diffusive flux was originally introduced in \cite{LY09, LY10} as part of the DDG method for solving the diffusion problem. 
If the flux parameters satisfy 
\begin{align*}
\beta_0\geq 1 \quad\text{ and } \quad \frac{1}{8}\leq \beta_1\leq \frac{1}{4},
\end{align*}
then the procedure developed in \cite{LY14} can be extended to the present setting to conclude that, under a suitable CFL condition, the simple Euler forward will keep the cell average  $ \bar u^n_j =\frac{\langle u_h^n\rangle_j}{\langle 1 \rangle_j} \in  [c_1, c_2]$ in each time step,  and the validity of the maximum principle when combined with a scaling limiter. 

For a two dimensional problem with $A = \left(\begin{array}{cc} a & c \\ c & b \end{array}\right)$, the DDG method on shape-regular Cartesian meshes with $\kappa^{-1} \leq
\frac{\Delta x}{\Delta y}\leq \kappa$ can be rendered MPS if
\begin{equation*}
\beta_0 \geq 1+\frac{\kappa|c| L(L-1)}{2 \min\{a, b\}} \quad\text{ and } \quad \frac{1}{8} \leq \beta_1 \leq \frac{1}{4}.
\end{equation*}
Here $L$ is the number of Gauss-Lobatto points used in the numerical evaluation of involved integrals. 
With $f(u)$ and $A = A(x,y,u)$, the MPS DDG schemes are analyzed in Section \ref{ssec2DNoninearCD} where the parameter range and the CFL conditions are established.

The main conclusion is as follows: by applying the weighted MPS limiter introduced in \cite{LY14}  to the DDG scheme designed here for \eqref{md}, with the time discretization by an SSP Runge-Kutta method (see \cite{GKS09}), we obtain a third  order accurate scheme solving (\ref{md}) satisfying the strict maximum principle in the sense that the numerical solution never goes out of the range $[c_1, c_2]$ as indicated by the initial data. 

\subsection{ Organization of the paper}
The rest of the paper is organized as follows. In Section \ref{sec1D}, we design the numerical method for one dimensional problems.  
We first formulate the  DDG scheme to solve the model heat equation, and prove the MPS property of the third order fully discretized DDG scheme, we then apply the result to show the MPS property for nonlinear convection-diffusion equations. 
Section \ref{sec2D} is organized similarly for two dimensional problems on Cartesian meshes. 
In Section \ref{secLimiter},  we present the MPS limiter, with which the algorithm is complete.  
In Section \ref{secNum}, numerical tests for the DDG method are reported, including examples from the heat equation, porous media equation 
the Buckley-Leverett equation, and two dimensional diffusion with the anisotropic diffusion.  
Concluding remarks are given in Section \ref{secConclude}.

\section{MPS schemes in one dimension}\label{sec1D}
We first investigate the MPS property for third order DDG schemes to weighted diffusion equations, and show how to apply the scheme obtained to nonlinear convection-diffusion equations. 
 
\subsection{The diffusion equation} We begin with the heat equation of the form 
\begin{align}\label{Ma}
M(x)\partial_t u=\partial_x (A(x)\partial_x u),
\end{align}
with $M(x)>0$ and $A(x)\geq 0$ on the spatial domain $\Omega$,  subject to initial data $u_0(x)$ and the periodic boundary condition.  It is known that  the following maximum principle holds:
\[
\text{if } c_1 \leq u_0(x) \leq c_2 \;\; \forall x \in \Omega, \text{ then } c_1 \leq u(t, x) \leq c_2 \;\; \forall x \in \Omega, t>0.
\]
In general, the problem can be either defined on a connected compact domain with proper boundary conditions, or it can involve the whole real line with solutions vanishing at the infinity.  
In our numerical scheme, we will always choose the spatial domain to be a connected interval. 
For simplicity, the periodic boundary conditions are applied. 

We partition the domain $\Omega$ by regular cells such that  $\Omega=\bigcup\limits_{j=1}^NI_j$ with $I_j = [x_{j-\frac12}, x_{j+\frac12}]$. 
Denote $h_j=x_{j+\frac12}-x_{j-\frac12}$ and $h=\max\limits_j h_j$.  
We seek numerical solutions in the discontinuous piecewise polynomial space, 
$$
V_h=\{v \in L^2(\Omega) |\quad v|_{I_j}\in P^k(I_j), \quad j=1\cdots, N\}.
$$
Here $P^k(I_j)$ is the space of $k$-th order polynomials on $I_j$. Note that the functions in $V_h$ can be double-valued at cell interfaces.  Hence notations $v^-$ and $v^+$ are used for the left limit and right limit of $v$.  The jump of these two values, $v^+-v^-$, is denoted by $[v]$, and the average by $\{v\}$. 

Throughout this paper we adopt the DDG numerical flux of the form 
\begin{align}\label{flux_DDG_1Dp}
\widehat{\partial_x v}& =\frac{\beta_0}{h_{j+\frac12}}[v] +\{\partial_x v\} +\beta_1 h_{j+\frac12}[\partial_x^2 v] \quad\text{ with } h_{j+\frac12} = \frac{h_j + h_{j+1}}{2}, 
\end{align}
when crossing the cell interface $x_{j+\frac12}$, and $(\beta_0, \beta_1)$ are in the range to be specified
so that the underlying third order scheme can weakly satisfy the maximum-principle.  
The parameter range was first identified in \cite{LY14} for a third order DDG scheme to feature the MPS property for linear Fokker-Planck equations.  

For model equation (\ref{Ma}), we consider a $(k+1)$th-order DG scheme: to find $u_h\in V_h$ such that for any test function $v\in V_h$,
\begin{align*}
\int_{I_j} M(x) \partial_t u_h v\,dx & =  -\int_{I_j} A(x) \partial_x u_h \partial_x v\,dx + \left. A \left[\widehat{\partial_x u_h}v +(u_h -\{u_h\})\partial_x v\right]\right|^{x_{j+\frac12}}_{x_{j-\frac12}},
\end{align*}
with the diffusive flux $\widehat{\partial_x u_h}$ as defined in (\ref{flux_DDG_1Dp}). 
This DDG scheme with interface correction was proposed in \cite{LY10} for the diffusion problem, as an improved version of that in \cite{LY09}. Based on the DDG scheme in \cite{LY09},  we will have 
\begin{align*}
\int_{I_j} M(x) \partial_t u_h v\,dx & =  -\int_{I_j} A(x) \partial_x u_h \partial_x v\,dx + \left. A \widehat{\partial_x u_h}v\right|^{x_{j+\frac12}}_{x_{j-\frac12}}.
\end{align*}

With the forward Euler time discretization, the weighted cell average for either of two DDG schemes evolves as
\begin{equation}\label{EF1Dv1}
 \langle u^{n+1}_h \rangle_j=  \langle u^n_h \rangle_j + \mu h\left. A \widehat{\partial_x u_h^n}\right|^{x_{j+\frac12}}_{x_{j-\frac12}},
\end{equation}
where $\mu = \frac{\tau}{h^2}$ is the mesh ratio.  For a concise presentation, a uniform mesh is assumed. Here and in what follows, we denote the time step length as $\tau$. 
Note that for periodic boundary conditions considered in this paper, we take  
$$
u_{h,\frac12}^-=u^-_{h, N+\frac12}, \quad u_{h, N+\frac12}^+=u^+_{h, \frac12}
$$
in the DDG numerical flux formula (\ref{flux_DDG_1Dp}) when $j=\frac12, N+\frac12$.

We also use the notation 
$$
 \langle q(\xi) \rangle_j= \frac{1}{2}\int_{-1}^1M(x_j+\frac{h}{2}\xi)q(\xi)d\xi, \text{ for any } q(\xi) \text{ on }[-1, 1].  
$$
And the cell average $\bar u_j$ is
$\bar u_j = \frac{\langle u_h \rangle_j}{ \langle 1 \rangle_j}$.
For $j=1,\cdots, N$, we define 
\begin{align}\label{dif1Dab}
a_j= \frac{\langle \xi-\xi^2 \rangle_j}{\langle 1-\xi \rangle_j}, \quad 
b_j=\frac{\langle \xi +\xi^2 \rangle_j}{\langle 1+\xi\rangle_j}, 
\end{align}
and
\begin{align}\nonumber
\tilde{\omega}^1_j = &\frac{\langle \gamma-\xi(1+\gamma)+\xi^2\rangle_j }{2(1+\gamma)},\\ \label{c5ome}
\tilde{\omega}^2_j =&\frac{\langle 1-\xi^2\rangle_j }{1-\gamma^2}, \\ \nonumber
\tilde{\omega}^3_j =&\frac{\langle -\gamma +\xi(1-\gamma)+\xi^2\rangle_j }{2(1-\gamma)}
\end{align}
with the weight $M(x)|_{I_j}=M(x_j+\frac{h}{2}\xi)$. We recall the following key result. 
\begin{lemma}\cite[Lemma 3.3]{LY14}  $\tilde{\omega}^i_j >0$ for $i=1, 2,3$ if and only if 
$$
\gamma \in (a_j, b_j),
$$
where $a_j, b_j$ satisfy $-1<a_j<b_j<1$. 
\end{lemma} 
\begin{proof} Here we only show $a_j< b_j$, which means that selection of $\gamma$ is always ensured. A direct calculation gives 
 \begin{align*}
b_j - a_j = 2\frac{\langle 1\rangle_j\langle\xi^2\rangle_j - \langle \xi\rangle_j^2}{\langle 1+\xi\rangle_j\langle 1-\xi \rangle_j} \geq 0,
\end{align*}
where the numerator can be reformulated as 
 \begin{align*}
& \int_{-1}^1\int_{-1}^1(\eta^2-\xi \eta)\tilde M(\xi) \tilde  M(\eta)d\xi d\eta +\int_{-1}^1\int_{-1}^1(\xi^2-\eta \xi)\tilde  M(\eta)\tilde  M(\xi)d\eta d\xi\\
=& \int_{-1}^1\int_{-1}^1(\xi-\eta)^2 \tilde  M(\xi)\tilde  M(\eta)d\xi d\eta >0, 
\end{align*}
where  $\tilde  M(\xi):= M(x_j+\frac{h}{2}\xi)$ has been used.
\end{proof} 

We thus have the following result.  
\begin{theorem}\label{thk2max}($k=2$) The scheme (\ref{EF1Dv1}) 
with
\begin{align}\label{betak2max}
\beta_0\geq 1
\quad \text{and} \quad  
\frac{1}{8}\leq \beta_1 \leq \frac{1}{4} 
\end{align}
is maximum-principle-satisfying, namely, $c_1\leq \bar{u}^{n+1}_j \leq c_2 $ if $u_h^n(x)$ is in $[c_1, c_2]$ on  $\{S_j\}_{j=1}^N$, where
$$
S_j = x_j +\frac{h}{2}\left\{-1, \gamma, 1\right\}
$$
with $\gamma$ satisfying
\begin{align}\label{abmax}
a_j <\gamma < b_j\quad \text{and} \quad |\gamma|\leq  8\beta_1-1,
\end{align}
under the CFL condition
$
\mu \leq \mu_0,
$
where $\mu_0$ is given in \eqref{CFL1Dk2} below. 
\end{theorem}
\begin{proof}  
{\sl Step 1.}  Weighted integral decomposition. 
Define 
$$
p_j(\xi)=u_h\left(x_j+\frac{h}{2}\xi\right) \text{ for } \xi\in [-1,1],
$$
we see that in the case of  ${p}_j(\xi)\in {P}^2[-1, 1]$, for any $\gamma \in (-1,1)$,  the unique interpolation of $p_j$ at three points $\{-1, \gamma, 1\}$ gives the following
 \begin{align}\label{pj}
{p}_j(\xi)=\frac{(\xi-1)(\xi -\gamma)}{2(1+\gamma)}{p}_j(-1)+\frac{(\xi -1)(\xi +1)}{(\gamma -1)(\gamma +1)}{p}_j(\gamma)+\frac{(\xi +1)(\xi -\gamma)}{2(1-\gamma)}{p}_j(1).
\end{align}
This yields the following identity for the weighted average,
\begin{equation}\label{c5wp}
    \langle u_h\rangle_j =\tilde{\omega}^1_j p_j(-1) + \tilde{\omega}^2_j p_j(\gamma) + \tilde{\omega}^3_j p_j(1),
\end{equation}
where $\tilde{\omega}^i_j$ given in (\ref{c5ome}) are ensured positive by Lemma 2.1.

{\sl Step 2. } Flux representation.  A direct calculation gives
\begin{align}\label{c5fluxk2gamma}
h \left.\widehat{\partial_xu_h}\right|_{x_{j+\frac12}}
=& \alpha_3(-\gamma) p_{j+1}(-1)+\alpha_2(-\gamma)p_{j+1}(\gamma)+\alpha_1(-\gamma)p_{j+1}(1)\\
&- \left[\alpha_1(\gamma) p_j(-1)+\alpha_2(\gamma)p_j(\gamma)+\alpha_3(\gamma)p_j(1)\right],\nonumber
\end{align}
where
\begin{align*}
\alpha_1(\gamma)= \frac{8\beta_1-1+\gamma}{2(1+\gamma)}, \quad
\alpha_2(\gamma)=2\frac{1-4\beta_1}{1-\gamma^2}, \quad
\alpha_3(\gamma)=\beta_0 +\frac{8\beta_1-3+\gamma}{2(1-\gamma)},
\end{align*}
are all positive due to (\ref{betak2max}) and \eqref{abmax}. 

{\sl Step 3.} Monotonicity under some CFL condition. We now substitute (\ref{c5wp}) and (\ref{c5fluxk2gamma}) into (\ref{EF1Dv1}) to obtain
$$
\langle u_h^{n+1}\rangle_j=R_j^n(M(\cdot), \mu, h, A)
$$
with 
\begin{align}\label{c5avedeck2}
R_j^n(M(\cdot), \mu, h, A)= & \langle u^n_h \rangle_j + \mu \left(\left.Ah\widehat{\partial_xu_h^n}\right|_{x_{j+\frac12}} - \left.Ah \widehat{\partial_xu_h^n}\right|_{x_{j-\frac12}}\right) \\
=&\left[\tilde{\omega}_j^1-\mu  \left(\alpha_3(-\gamma) A_{j-\frac12} +\alpha_1(\gamma) A_{j+\frac12} \right) \right]p_j(-1) \nonumber\\
&+\left[\tilde{\omega}_j^2 - \mu  \left(\alpha_2(-\gamma) A_{j-\frac12} +\alpha_2(\gamma) A_{j+\frac12}\right)\right]p_j(\gamma)\nonumber\\
&+\left[\tilde{\omega}_j^3 - \mu  \left(\alpha_1(-\gamma) A_{j-\frac12} +\alpha_3(\gamma) A_{j+\frac12} \right) \right]p_j(1) \nonumber\\
&+\mu  A_{j+\frac12}\left[\alpha_3(-\gamma) p_{j+1}(-1)+\alpha_2(-\gamma)p_{j+1}(\gamma)+\alpha_1(-\gamma)p_{j+1}(1) \right]\nonumber \\
&+\mu  A_{j-\frac12}\left[\alpha_1(\gamma) p_{j-1}(-1)+\alpha_2(\gamma)p_{j-1}(\gamma)+\alpha_3(\gamma)p_{j-1}(1) \right].\nonumber
\end{align}
Here it is understood that $p_0(\xi):=p_N(\xi+|\Omega|)$ and $p_{N+1}(\xi+|\Omega|)=p_1(\xi)$ for incorporating the periodic boundary conditions. Note also that $\sum_{i}^3\alpha_i(\gamma)=\beta_0= \sum_{i}^3\alpha_i(-\gamma)$.   
Using the fact  $ \tilde{\omega}_j^3(\gamma)=\tilde {\omega}_j^1(-\gamma)$ and formulas for
$\alpha_2(\gamma)$,  $\tilde{\omega}_j^2$, we see that if  $\mu$ is chosen to be smaller than $\mu_0$ where 
\begin{align}\label{CFL1Dk2}
\mu_0 = 
	\left(\max_{1\leq j\leq N} A(x_{j+\frac12})  \right)^{-1}
	\min_{1\leq j\leq N}
	\left\{ \frac{{\tilde\omega}_j^1(\pm \gamma)}{\alpha_3(\mp \gamma) 
+\alpha_1(\pm \gamma) },  \frac{\langle 1-\xi^2\rangle_j}{4(1-4\beta_1)}
\right\},
\end{align}
(\ref{c5avedeck2}) is nondecreasing in the point values $p_j(\pm1), p_j(\gamma), p_{j\pm 1}(\pm1), p_{j\pm 1}(\gamma)$, hence when these values are replaced with the lower and upper bounds $c_1$ and $c_2$ respectively, we have
\[
c_1\sum_{i=1}^3\tilde{\omega}_j^i \leq \langle u_h^{n+1}\rangle_j \leq c_2 \sum_{i=1}^3\tilde {\omega}_j^i,
\]
since the terms with $\alpha_i$'s are cancelled out. Moreover the sum of $\tilde{\omega}_j^i$ is $\langle 1\rangle_j$. Therefore
\[
c_1\langle 1\rangle_j \leq \langle u_h^{n+1}\rangle_j \leq c_2\langle 1\rangle_j \quad \Rightarrow \quad c_1 \leq \bar{u}_j^{n+1} \leq c_2.
\]
\end{proof}

\begin{remark} 
For other types of boundary conditions, the boundary flux needs to be modified.  A  similar result to Theorem \ref{thk2max} may be established as long as the PDE problem satisfies a maximum principle.
\end{remark}
\begin{remark}
The use of $\gamma$ is essential for the success of our schemes, in particular when $M(x)$ is a function.  
For the special case $M(x)=1$, we have $\gamma\in (-\frac13, \frac13)$, hence $\gamma=0$, which corresponds to the usual Gauss quadrature point, is admissible.  
The CFL number given in \eqref{CFL1Dk2} may be optimized by carefully tuning $\gamma \in (a_j, b_j)$, but not in a linear fashion.  
Nevertheless, it was observed in \cite{LY14} that  the larger $|\gamma|$ is, the better the scheme's performance for the heat equation. 
\end{remark}

\subsection{Application to nonlinear convection-diffusion equation}
In this section we will demonstrate how to apply our MPS DG method in section 2.1 to the nonlinear convection-diffusion equation, 
$$
\partial_t u +\partial_x f(u) =\partial_x (A(x, u)\partial_x u),
$$
where $f(u)$ is a smooth function and diffusion coefficient $A(x, u)\geq 0$, subject to initial data $u(0, x)=u_0(x)$, and periodic boundary conditions.  By applying the DG approximation, we obtain the following scheme. We seek $u_h\in V_h$ such that for any test function $v\in V_h$,
\begin{align}\label{semi1D}
\int_{I_j} \partial_t u_h v\,dx & = \int_{I_j} f(u_h)\partial_xv\,dx -\hat f(u_h^-, u_h^+) v\Big |^{x_{j+\frac12}}_{x_{j-\frac12}}\\ \notag
& \quad -\int_{I_j} A_h \partial_x u_h \partial_x v\,dx + \left.\{A_h\}\left[\widehat{\partial_x u_h}v +(u_h -\{u_h\})\partial_x v\right]\right|^{x_{j+\frac12}}_{x_{j-\frac12}}.
\end{align} 
For diffusion part,  we adopt the DDG diffusive flux \eqref{flux_DDG_1Dp} 
For convection, any monotone numerical flux can be used, i.e., $\hat f(u, v)$ is Lipschitz continuous, nondecreasing in $u$ and nonincreasing in $v$, consistent with $f(u)$ in the sense that $ \hat f(u, u)=f(u)$. For example, the global Lax-Friedrichs flux 
\begin{equation*}
\hat f(u_h^-, u^{+}_h)=\frac{1}{2} (f(u^{-}_h) +f(u^{+}_h) -\sigma (u^{+}_h-u^{-}_h)), 
\quad \sigma=\max_{u \in [c_1, c_2]}|f'(u)|.
\end{equation*}

We consider the first order Euler forward temporal discretization of (\ref{semi1D}) to obtain 
\begin{align}\label{fully_1D}
\int_{I_j} \frac{u_h^{n+1}-u_h^n}{\tau} v\,dx & = \int_{I_j} f(u^n_h)\partial_xv\,dx 
	- \hat f((u_h^n)^-, (u_h^n)^+) v\Big |^{x_{j+\frac12}}_{x_{j-\frac12}}\\ \notag
& \quad -\int_{I_j} A_h^n \partial_x u^n_h \partial_x v\,dx + \left.\{A_h^n\}\left[\widehat{\partial_x u^n_h}v +(u^n_h -\{u^n_h\})\partial_x v\right]\right|^{x_{j+\frac12}}_{x_{j-\frac12}}.
\end{align}
By taking the test function $v = 1$ on $I_j$ and $0$ elsewhere, we obtain the evolutionary update for the cell average, 
\begin{equation*}
 \bar u^{n+1}_j=  \bar u^n_j - \lambda  \hat f^n \Big|^{x_{j+\frac12}}_{x_{j-\frac12}}  
 + \mu h \{A_h^n\} \widehat{\partial_xu_h^n} 
 \Big|^{x_{j+\frac12}}_{x_{j-\frac12}},
\end{equation*}
where $\lambda = \frac{\tau}{h}$  and $\mu = \frac{\tau}{h^2}$ are the mesh ratios.  Assuming that $ \bar u^n_j \in [c_1, c_2]$ for all $j$'s, we would like to derive some sufficient conditions such that 
$ \bar u^{n+1}_j \in [c_1, c_2] $ under certain restrictions on $\lambda$ and $\mu$. 

For piecewise quadratic polynomials, the main result can be stated as follows.  
\begin{theorem}\label{thk2_1D}($k=2$) The scheme (\ref{fully_1D}) with
\begin{align*}
\beta_0\geq 1
\quad \text{and} \quad  
\frac{1}{8}\leq  \beta_1 \leq  \frac{1}{4} 
\end{align*}
is maximum-principle-satisfying; namely, $\bar{u}_{j}^{n+1}\in [c_1, c_2]$  if $u_h^n(x) \in [c_1, c_2]$ on the set $S_j$'s where
$$
S_j = x_j +\frac{h}{2}\left\{-1, \gamma, 1\right\}
$$
with $\gamma$ satisfying
\begin{align}\label{ab}
-\frac{1}{3} <\gamma <  \frac{1}{3} \quad \text{and} \quad |\gamma|\leq  8\beta_1-1,  
\end{align}
under the CFL condition
\begin{align*}
\lambda \leq \lambda_0, \quad \mu \leq \mu_0
\end{align*}
for some $\lambda_0$ and $\mu_0$ defined in (\ref{CFL1Dk2-}) and  (\ref{CFL1Dk2+}), respectively. 
\end{theorem}
\begin{proof} We present the proof in four steps:  \\
{\sl Step 1.}  Split: we split the average $\bar u_j^n$ into two halves so that   
\begin{equation*}
\bar u_j^{n+1}=\frac{1}{2} \mathcal{C}_j^n +\frac{1}{2}\mathcal{D}_j^n,
\end{equation*}
where the convection term is 
\begin{equation}\label{cj}
\mathcal{C}_j=\bar u_j -2\lambda \hat f \Big |^{x_{j+\frac12}}_{x_{j-\frac12}}
\end{equation}
and the diffusion term is 
\begin{equation*}
\mathcal{D}_j =\bar u_j + 2\mu h \{A_h^n\} \widehat {\partial_xu_h}\Big |^{x_{j+\frac12}}_{x_{j-\frac12}}.
\end{equation*}
This split is for convenient presentation and may not lead to an optimal CFL condition. 

{\sl Step 2.} The integral decomposition. From (\ref{c5wp})  it follows 
\begin{equation*}
\bar u_j = \omega^1 p_j(-1)+\omega^2 p_j(\gamma)+\omega^3 p_j(1),
\end{equation*}
where $\omega^i=\tilde \omega^i, \langle u_h\rangle_j = \bar u_j$ for $M(x)\equiv1$, and  
\begin{align*}
\omega^1 = \frac{1+3\gamma }{6(1+\gamma)},\; \;
\omega^2 = \frac{2}{3(1-\gamma^2)}, \;
\omega^3 = \frac{1-3\gamma}{6(1-\gamma)}.
\end{align*}
These coefficients are positive for $\gamma$ satisfying (\ref{ab}). \\

{\sl Step 3. }  The convection term.

Using the cell average decomposition and the flux formula,  we rewrite (\ref{cj}) as  
\begin{align*}
\mathcal{C}_j = & \bar u^n_j -2 \lambda\left(\left.\widehat{f}\right|_{x_{j+\frac12}} - \left.  \widehat{f}\right|_{x_{j-\frac12}}\right) \\
=& \omega^3 p_j(1) - 2\lambda \hat f(p_j(1), p_{j+1}(-1)) + \omega^2 p_j(\gamma)  + \omega^1 p_j(-1) 
+ 2\lambda \hat f(p_{j-1}(1), p_j(-1)) \nonumber \\
= &: G(p_{j-1}(1), p_j(-1), p_j(\gamma), p_j(1), p_{j+1}(-1)). \nonumber
\end{align*}
For a monotone flux $\hat f(u, v)$ being Lipschitz continuous with Lipschitz constant $\mathcal{L}$, 
$G$ is non-increasing in all the four arguments provided the following condition is met, 
$$
2 \lambda \mathcal{L} \leq \min\{ \omega^1, \omega^3\}.
$$ 
Note that for the Lax-Friedrichs flux, $\mathcal{L} = \max\limits_{u\in[c_1,c_2]} |f'(u)|$.  
Moreover, the consistency of the flux $\hat f(u, u)=f(u)$ implies that $G(u, u, u, u)=u$.  
Hence we have 
$$
\mathcal{C}_j \in [G(c_1, c_1, c_1, c_1, c_1), G(c_2, c_2, c_2, c_2, c_2)]=[c_1, c_2],  
$$
as long as the involved values are in $[c_1, c_2]$. It suffices to take 
\begin{align}\label{CFL1Dk2-}
\lambda_0= \frac{1}{2\mathcal{L}}  \min\{ \omega^1, \omega^3\}=\frac{1-3\gamma}{12\mathcal{L}(1-\gamma)},
\end{align}
where we have used the fact that $\omega^1(\gamma)=\omega^3(-\gamma)$.

{\sl Step 4.} The  diffusion term. We apply the result in Theorem \ref{thk2max} to the case with $M(x)\equiv1$ and $\mu$ replaced by $2\mu$ to conclude that 
$\mathcal{D}_j \in [c_1, c_2]$ if $u_h^n(x) \in [c_1, c_2]$ on $S_j$ and $\mu \leq \mu_0$ with 
\begin{align}\label{CFL1Dk2+}
\mu_0 & = \frac{1}{2}\left(\max_j \{A_h^n\}_{j+\frac12}  \right)^{-1}\min_{1\leq j\leq N}\left\{ \frac{{\omega}^1(\pm \gamma)}{\alpha_3(\mp \gamma) 
+\alpha_1(\pm \gamma) },  \frac{1}{3(1-4\beta_1)}
\right\}\\ \notag
& =\frac{1}{12}\left(\max_{1\leq j\leq N} \{A_h^n\}_{j+\frac12}\right)^{-1}\min_{1\leq j\leq N}\left\{ \frac{1\pm 3\gamma}{\beta_0(1\pm \gamma)+8\beta_1-2}, \frac{2}{1-4\beta_1} \right\}.
\end{align}
\end{proof}

The above analysis can be readily carried over to (\ref{md}) in one dimension, i.e.,   
\begin{align}\label{md+}
M(x)\partial_t u +\partial_x f(u) =\partial_x (A(x, u)\partial_x u). 
\end{align}
We summarize the result in the following. 
\begin{theorem}\label{thk2_1D_md}($k=2$) The scheme \eqref{fully_1D} when applied to \eqref{md+} with
\begin{align*}
\beta_0\geq 1
\quad \text{and} \quad  
\frac{1}{8}\leq  \beta_1 \leq  \frac{1}{4} 
\end{align*}
is maximum-principle-satisfying; namely, $\bar{u}_{j}^{n+1}\in [c_1, c_2]$  if $u_h^n(x) \in [c_1, c_2]$ on the set $S_j$'s where
$$
S_j = x_j +\frac{h}{2}\left\{-1, \gamma, 1\right\}
$$
with $\gamma$ satisfying
\begin{align*}
\gamma \in (a_j, b_j) \; \text{as in \eqref{dif1Dab} and} \quad |\gamma|\leq  8\beta_1-1,  
\end{align*}
under the CFL condition
\begin{align*}
\lambda \leq \lambda_0, \quad \mu \leq \mu_0
\end{align*}
for some $\lambda_0$ and $\mu_0$ defined in \eqref{CFL1Dk2md_conv} and  \eqref{CFL1Dk2md_diff}, respectively. 
\end{theorem}
\begin{proof} 
The proof is entirely analogous to the proof of Theorem \ref{thk2_1D}. One only needs to replace $\omega^{i}$ by $\tilde\omega^{i}$ for $i=1, 2, 3$, and  $\bar u_j$ by $\langle u_h\rangle_j$, respectively, so to obtain different CFL conditions.  More precisely,  (\ref{CFL1Dk2-}) in Step 3 needs to be  replaced by 
\begin{align}\label{CFL1Dk2md_conv}
\lambda_0= \frac{1}{2\mathcal{L}}  \min\{ \tilde\omega^1(\gamma), \tilde\omega^3(\gamma)\}= \frac{1}{2\mathcal{L}}  \min\{ \tilde\omega^1(\pm\gamma)\}, 
\end{align}
and (\ref{CFL1Dk2}) by a factor of $1/2$ gives    
\begin{align}\label{CFL1Dk2md_diff}
\mu_0 = \frac{1}{2}
	\left(\max_{1\leq j\leq N} A(x_{j+\frac12})  \right)^{-1}
	\min_{1\leq j\leq N}
	\left\{ \frac{{\tilde\omega}_j^1(\pm \gamma)}{\alpha_3(\mp \gamma) 
+\alpha_1(\pm \gamma) },  \frac{\langle 1-\xi^2\rangle_j}{4(1-4\beta_1)}
\right\}.
\end{align}
\end{proof} 

\section{MPS schemes in two dimensions}\label{sec2D}
In this section, we design an MPS DG scheme to solve two dimensional problems on Cartesian meshes.
Consider the model equation
$$
M(x,y)\partial_t u = \nabla\cdot(A\nabla u) \text{ with } 
A = \left(\begin{array}{cc} a & c \\ c & b \end{array}\right), \quad \text{ for } (x, y)\in \Omega \subset \mathbb{R}^2, \quad t>0,
$$
subject to the initial data $u_0(x, y)$ and periodic boundary conditions.  Here $M(x, y) > 0$ is a given function, 
$a, b$ and $c$ are constant parameters so that $A$ is nonnegative definite.
The domain $\Omega =I \times J$ is a rectangle given by two intervals $I$ and $J$ in $x$ and $y$ direction, respectively.
Let 
$\cup_{i=1}^{N_x}\cup_{j=1}^{N_y} K_{ij} $ be a partition of the domain $\Omega$, with $K_{ij}=I_i\times J_j$, where 
$$
I_{i}=[x_{i-\frac12}, x_{i+\frac12}], \quad J_j = [y_{j-\frac12}, y_{j+\frac12}].
$$
The finite element space is defined as 
$$
V_h=\{v\in L^2(\Omega),  v\big|_{K_{ij}}\in Q^k(K_{ij}), i=1, \cdots, N_x, j=1,\cdots, N_y\}.
$$
Here $Q^k(K_{ij})$ is the tensor product space of $P^k(I_i)$ and $P^k(J_j)$.   Hence the DDG scheme can be formulated 
as follows: to find $u_h\in V_h$ such that for all $v\in V_h$, 
\begin{align}\label{EFfully2D-}
\int_{K_{ij}} Mu_{h}^{n+1}v\,dxdy 
= &\int_{K_{ij}} Mu_{h}^{n}v\,dxdy 
- \tau \int_{K_{ij}} A\nabla u_h^n\cdot\nabla v\,dxdy \notag \\
 &+ \tau \int_{\partial K_{ij}}A\widehat{\nabla u_h^n}\cdot \nu v\,ds 
 + \tau \int_{\partial K_{ij}}A{\nabla v}\cdot \nu (u_h^n - \{u_h^n\})\,ds, 
\end{align}
where $\nu$ is the outward unit normal to the cell boundary $\partial K_{ij}$, and the numerical flux 
\begin{equation}\label{df2}
\widehat{\nabla u_h}= (\widehat{ \partial_x u_h}, \widehat{\partial_y u_h})^\top
\end{equation}
is defined as follows, 
\begin{align}\label{beta2D}
\left.\widehat{\partial_x u_h}\right|_{(x_{i+\frac12},y)}& =\frac{\beta_{0}}{\Delta x}[u_h] +\{\partial_x u_h\} +\beta_{1} \Delta x [\partial_x^2 u_h], \;  
\left.\widehat{\partial_y u_h}\right|_{(x_{i+\frac12},y)}  =\{\partial_y u_h\}, \notag\\
\left.\widehat{\partial_y u_h}\right|_{(x,y_{j+\frac12})}& =\frac{\beta_{0}}{\Delta y}[u_h] +\{\partial_y u_h\} +\beta_{1} \Delta y [\partial_y^2 u_h], \;
\left.\widehat{\partial_x u_h}\right|_{(x,y_{j+\frac12})} =\{\partial_x u_h\}, \notag
\end{align}
where $\beta_0, \beta_1$ are flux parameters to be determined to ensure the desired MPS property.  
Note that in $\widehat{\partial_y u_h}|_{(x_{i+\frac12},y)}$, jump terms do not show up since along interface 
 $x = x_{i+\frac12}$ and $y \in [y_{j-\frac12}, y_{j+\frac12}]$, there is no jump of polynomials  in $y$ direction. This argument applies to $\partial_x u_h|_{(x,y_{j+\frac12})}$ as well.
Here for a concise expression of the numerical flux, a uniform mesh has been used, with $\Delta x = x_{i+\frac12}-x_{i-\frac12}$ and $\Delta y=y_{j+\frac12}-y_{j-\frac12}$.  

To proceed, we  recall some conventions similar to the one-dimensional case.  The weighted cell average is defined as 
\begin{equation*}
\langle u^{n+1}\rangle_{ij} = \frac{\int_{K_{ij}}Mu_h\,dxdy}{\Delta x \Delta y}
	= \dashint_{J_j} \dashint_{I_i}Mu_h\,dxdy,
\end{equation*}
where $\dashint$ denotes the average integral.  We also define the weighted interval average in $x$ and $y$, respectively,
$$
\langle \phi(\xi)  \rangle_i(y)=\dashint_{-1}^1M\left(x_i+\frac{\Delta x}{2}\xi, y\right)\phi(\xi)d\xi, \quad
\langle \phi(\eta) \rangle_j(x)=\dashint_{-1}^1M\left(x, y_j+\frac{\Delta y}{2}\eta\right)\phi(\eta)d\eta.
$$
The cell average 
$$
\overline{u}_{ij} = \frac{\int_{K_{ij}}Mu_h\,dxdy}{\int_{K_{ij}}M\,dxdy}=\frac{\langle u_h\rangle_{ij}}{\langle 1\rangle_{ij}}
$$
update can be obtained from  (\ref{EFfully2D-}) as
\begin{align}\label{EF2Dv1}
\langle u^{n+1}_h\rangle_{ij} =& \langle u^{n}\rangle_{ij} 
+ \mu_x\Delta x \left.\dashint_{J_j}\big(a\widehat{\partial_x u_h^n} + c\widehat{\partial_y u_h^n}\big)\,dy\right|_{\partial I_i}
+ \mu_y\Delta y \left.\dashint_{I_i}\big(b\widehat{\partial_y u_h^n} + c\widehat{\partial_x u_h^n}\big)\,dx\right|_{\partial J_j},
\end{align}
where $\mu_{x} = \frac{\tau }{(\Delta x)^2}$ and $\mu_{y} = \frac{\tau}{(\Delta y)^2}$. 
Let $\mu = \mu_x + \mu_y$ and decompose $\langle u^{n}\rangle_{ij}$ as
\[
\langle u^{n}\rangle_{ij} = \frac{\mu_x}{\mu} \langle u^{n}\rangle_{ij}  + \frac{\mu_y}{\mu} \langle u^{n}\rangle_{ij},
\] 
so that (\ref{EF2Dv1}) can be rewritten as 
\begin{align}\label{EF2DH}
\langle u^{n+1}\rangle_{ij} =& \frac{\mu_x}{\mu}\dashint_{J_j}H_1(y)\,dy + \frac{\mu_y}{\mu}\dashint_{I_i} H_2(x)\,dx  +B,
\end{align}
where 
\begin{align*}
& H_1(y)  =\dashint_{I_i}M(x,y)u_{h}^n\,dx
	+ \mu\Delta x\left. a\widehat{\partial_x u_h^n}\right|_{\partial I_i},\\
& H_2(x)  =\dashint_{J_j}M(x,y)u_{h}^n\,dy
	+ \mu\Delta y\left. b\widehat{\partial_y u_h^n} \right|_{\partial J_j},\\
& B =  \frac{c\tau}{|K_{ij}|} \left[ \int_{J_j}\{\partial_y u_h\}dy |_{\partial I_i} +\int_{I_i}\{\partial_x u_h\}dx |_{\partial J_j} \right].
\end{align*}
Notice that $B$ can be expressed as a combination of point values of $u_h^n$ at four vertices of $K_{ij}$ as 
\begin{align*}
B& =\frac{c\tau}{2\Delta x\Delta y}\big(2u_h^n(x_{i+\frac12}^-, y_{j+\frac12}^-) - 2u_h^n(x_{i+\frac12}^-, y_{j-\frac12}^+)
-2u_h^n(x_{i-\frac12}^+, y_{j+\frac12}^-) + 2u_h^n(x_{i-\frac12}^+, y_{j-\frac12}^+)\big) \\
& + \frac{c\tau}{2\Delta x\Delta y}\big(u_h^n(x_{i+\frac12}^+, y_{j+\frac12}^-) +u_h^n(x_{i+\frac12}^-, y_{j+\frac12}^+)  - u_h^n(x_{i+\frac12}^-, y_{j-\frac12}^-)
-u_h^n(x_{i+\frac12}^+, y_{j-\frac12}^+)\big)\\
& + \frac{c\tau}{2\Delta x\Delta y}\big(u_h^n(x_{i-\frac12}^+, y_{j-\frac12}^-)+u_h^n(x_{i-\frac12}^-, y_{j-\frac12}^+) -u_h^n(x_{i-\frac12}^+, y_{j+\frac12}^+) 
-u_h^n(x_{i-\frac12}^-, y_{j+\frac12}^-) \big).
\end{align*}
The two integrals in (\ref{EF2DH}) can be approximated by the Gauss-Lobatto quadrature rule with sufficient accuracy. 
Let us assume that we use an $L-$point Gauss-Lobatto quadrature rule with $L\geq \frac{k+3}{2}$ points, which has accuracy of at least $O(h^{k+2})$.
Let
\begin{align*}
&\hat S_{i}^x = \{x_{i-\frac12} = \hat x^1_i < \cdots < \hat x_i^{\sigma} < \cdots < \hat x^L_i=x_{i+\frac12}\}, \notag\\
&\hat S_{j}^y = \{y_{j-\frac12} = \hat y^1_j < \cdots < \hat y_j^{\sigma} < \cdots < \hat y^L_j=y_{j+\frac12}\}
\end{align*}
denote the quadrature points on $I_i$ and $J_j$, respectively, 
and $\hat\omega^{\sigma}$'s be the associated quadrature weights so that
$$
\sum_{\sigma=1}^L \hat \omega^{\sigma}=1.
$$
Using the quadrature rule on the right-hand side of (\ref{EF2DH}), we obtain the following scheme
\begin{equation}\label{c5EF2DHQ}
\langle u^{n+1}_h \rangle_{ij}= \frac{\mu_x}{\mu}\sum_{\sigma=1}^L \hat\omega^{\sigma}H_1(\hat y_j^\sigma) + \frac{\mu_y}{\mu}\sum_{\sigma=1}^L \hat\omega^{\sigma}H_2(\hat x_i^\sigma)+B. 
\end{equation}
Also set 
\begin{align*} 
 S^x_i & = \{x_{i-\frac12}, x_i^{\gamma}, x_{i+\frac12}\} = x_i + \frac{\Delta x}{2}\{-1, \gamma^x, 1\}, \\
 S^y_j & = \{y_{j-\frac12}, y_j^{\gamma}, y_{j+\frac12}\} = y_j + \frac{\Delta y}{2}\{-1, \gamma^y, 1\}
\end{align*}
 to denote the test sets on $[x_{i-\frac12}, x_{i+\frac12}]$ and $[y_{j-\frac12}, y_{j+\frac12}]$, respectively,  with $\gamma^x, \gamma^y$ satisfying
\begin{align*} \label{c5angle2D}
& \frac{\langle \xi-\xi^2\rangle_i}{\langle 1- \xi \rangle_i} (y_j^\sigma) < \gamma^x < \frac{\langle \xi +\xi^2\rangle_i}{\langle 1+\xi \rangle_i} (y_j^\sigma), \quad |\gamma^x|\leq 8\beta_1-1,\\
&  \frac{\langle \eta-\eta^2\rangle_j} {\langle 1- \eta \rangle_j} (x_i^\sigma) < \gamma^y < \frac{\langle \eta +\eta^2\rangle_j}{\langle 1+\eta \rangle_j} (x_i^\sigma), \quad  |\gamma^y|\leq 8\beta_1-1.\nonumber
\end{align*}
We use $\otimes$ to denote the tensor product and define
\begin{equation*}\label{c5set}
S_{ij}=(S_i^x\otimes\hat S^y_j)\cup (\hat S_i^x \otimes S_j^y).
\end{equation*}
The main result can now be stated in the following.
\begin{theorem}\label{th2dk2}($k=2$)
Consider the two dimensional DDG scheme (\ref{EFfully2D-}) on rectangular meshes.
Assume the mesh is regularly shaped, i.e., $ \kappa^{-1} \leq\frac{\Delta x}{\Delta y} \leq \kappa 
$,  for some constant $\kappa > 0$, and the flux parameters $(\beta_0, \beta_1)$  satisfy 
\begin{equation}\label{beta2dp}
\beta_0 \geq 1+\frac{\kappa|c|}{2 \hat \omega^1 \min\{a, b\}}.
\quad\text{ and } \quad 
\frac{1}{8} \leq \beta_1 \leq \frac{1}{4} 
\end{equation}
If $u_h^n(x, y)\in [c_1, c_2]$ for all $(x, y)\in S_{ij}$, then there exists $\mu_0>0$ such that if $\mu \leq \mu_0$ the cell average
$
\bar u^{n+1}_{ij}  \in [c_1, c_2]. 
$ 
More precisely, we have 
\begin{align}\label{CFL2d-}
 \mu_0 = \min_{i,j}\underline\omega_{ij}
	\min\left\{\frac{\hat\omega^1}{\hat\omega^1\max\{a,b\}\big(\beta_0 + \frac{8\beta_1-2}{1+\gamma}\big) + \kappa|c|},
	\frac{1-\gamma^2}{4\max\{a,b\}(1-4\beta_1)}\right\},
\end{align}
where $\underline\omega_{ij}>0$ is defined in equation \eqref{min_omega}, and $\gamma = \max\{|\gamma^x|, |\gamma^y|\} \leq 8\beta_1-1$.
\end{theorem}

The proof is relegated to Appendix A.  We proceed to deal with nonlinear diffusion equations in the next  subsection.

\subsection{Application to nonlinear diffusion equations} 
\label{ssec2DNoninearCD}
This section is devoted to application to nonlinear diffusion equations of the form 
$$
\partial_t u = \nabla\cdot(A\nabla u) \text{ with } 
A(x,y,u) = \left(\begin{array}{cc} a & c \\ c & b \end{array}\right) \quad \text{ for } (x, y)\in \Omega \subset \mathbb{R}^2, \quad t>0,
$$
subject to initial data $u_0(x, y)$ and periodic boundary conditions, and $A(x,y,u)$ is nonnegative definite. 
This type of model arises in a wide range of applications. 

Hence the DDG scheme can be formulated as follows: to find $u_h\in V_h$ such that for all $v\in V_h$, 
\begin{align}\label{EFfully2D}
\int_{K_{ij}} u_{h}^{n+1}v\,dxdy 
= &\int_{K_{ij}} u_{h}^{n}v\,dxdy 
- \tau \int_{K_{ij}} A_h^n \nabla u_h^n\cdot\nabla v\,dxdy \notag \\
 &+ \tau \int_{\partial K_{ij}} \{A_h^n\}\widehat{\nabla u_h^n}\cdot \nu v\,ds 
 + \tau \int_{\partial K_{ij}}\{A_h^n\} {\nabla v}\cdot \nu (u_h^n - \{u_h^n\})\,ds, 
\end{align}
where $\nu$ is the outward unit normal to the cell boundary $\partial K_{ij}$, the numerical flux 
$
\widehat{\nabla u_h}
$
is defined in (\ref{df2}), and $A_h^n=A(x, y, u_h^n)$. 

The  cell average evolves according to 
$$
\bar u^{n+1}_{ij} = \bar u^{n}_{ij} 
+ \mu_x\Delta x \left.\dashint_{J_j}\big(\{a_h^n\}\widehat{\partial_x u_h^n} + \{c_h^n\}\widehat{\partial_y u_h^n}\big)\,dy\right|_{\partial I_i}
+ \mu_y\Delta y \left.\dashint_{I_i}\big(\{b_h^n\}\widehat{\partial_y u_h^n} + \{c_h^n\}\widehat{\partial_x u_h^n}\big)\,dx\right|_{\partial J_j}.
$$
That is 
\begin{align}\label{EF2Dv1+}
\bar u^{n+1}_{ij}  =  \frac{\mu_x}{\mu}\dashint_{J_j}H_1(y)\,dy + \frac{\mu_y}{\mu}\dashint_{I_i} H_2(x)\,dx  +B, 
\end{align}
where  $\mu_{x} = \frac{\tau }{(\Delta x)^2}$,  $\mu_{y} = \frac{\tau}{(\Delta y)^2}$ and $\mu = \mu_x + \mu_y$,  with 
\begin{align*}
& H_1(y)  =\dashint_{I_i}u_{h}^n\,dx
	+ \mu\Delta x\left. \{a_h^n\}\widehat{\partial_x u_h^n}\right|_{\partial I_i},\\
& H_2(x)  =\dashint_{J_j}u_{h}^n\,dy
	+ \mu\Delta y\left. \{b_h^n\}\widehat{\partial_y u_h^n} \right|_{\partial J_j},\\
& B =  \frac{\tau}{|K_{ij}|} \left[ \int_{J_j}\{c_h^n\} \{\partial_y u_h\}dy |_{\partial I_i} +\int_{I_i}\{c_h^n\}\{\partial_x u_h\}dx |_{\partial J_j} \right].
\end{align*}

The main result  can be stated in the following.
\begin{theorem}\label{th2dk2+}($k=2$) Consider the two dimensional DDG scheme (\ref{EFfully2D}) on rectangular meshes.
Assume the mesh is regularly shaped, i.e., $ \kappa^{-1} \leq\frac{\Delta x}{\Delta y} \leq \kappa 
$,  for some constant $\kappa > 0$, and the flux parameters $(\beta_0, \beta_1)$  satisfy 
\begin{equation}\label{beta2d}
\beta_0 \geq 1+\frac{2\kappa \|c\|_\infty L(L-1) }{(1-\gamma) \min\{a, b\}},
\quad\text{ and } \quad 
\frac{1}{8} \leq \beta_1 \leq \frac{1}{4} 
\end{equation}
and $\gamma = \max\{|\gamma^x|, |\gamma^y|\} \leq 8\beta_1-1$,
where
\[
\|c\|_\infty = \max_{(x,y) \in\Omega, u\in[c_1, c_2]}|c(x,y,u)|, \quad \min\{a, b\} = \min_{(x,y) \in \Omega, u\in[c_1, c_2]}\big\{a(x,y,u), b(x,y,u)\big\}.
\]
If $u_h^n(x, y)\in [c_1, c_2]$ for all $(x, y)\in S_{ij}$, then there exists $\mu_0>0$ such that if $\mu \leq \mu_0$ the cell average
$
\bar u^{n+1}_{ij}  \in [c_1, c_2]. 
$ 
More precisely, we have 
\begin{align}\label{CFL2d}
 \mu_0 = \min\left\{\frac{1-3\gamma}{6\max\{a,b\}\big(\beta_0 + \frac{8\beta_1-2}{1+\gamma}\big)(1-\gamma) + 12\kappa\|c\|_\infty L(L-1)},
	\frac{1}{6\max\{a,b\}(1-4\beta_1)}\right\},
\end{align}
where 
\[
\max\{a, b\} = \max_{(x,y) \in\Omega, u\in[c_1, c_2]}\big\{a(x,y,u), b(x,y,u)\big\}.
\]
\end{theorem}
The proof is relegated to Appendix B. 

\section{Scaling limiter and the MPS algorithm}\label{secLimiter}
\subsection{Scaling limiter}  The one dimensional result in Theorem \ref{thk2max} and the two dimensional result  in Theorem \ref{th2dk2}
tell us that for the DDG scheme with forward Euler time discretization, we need to modify $u_h^n$ such that it is in $[c_1, c_2]$ on the test set $S=S_j$ or $S_{ij}$.  
In one dimensional case, we can use the following  scaling limiter
\begin{equation}\label{reconMax}
\tilde{u}_h(x) = \theta \left(u_h(x) - \bar{u}_j\right) + \bar{u}_j \quad 
\text{ with } \theta = \min\left\{1,  \left|\frac{\bar{u}- c_1}{\bar{u}_j - m_1}\right|,  
\left|\frac{c_2-\bar{u}_j}{m_2- \bar{u}_j}\right| \right\},
\end{equation}
where
\begin{equation}\label{zetaMax}
m_1 = \min_{x\in S_j}u_h(x), \qquad m_2 = \max_{x\in S_j}u_h(x).
\end{equation}
In the two dimensional case, 
\begin{equation}\label{recon2D}
\tilde{u}_h(x, y) = \theta \left(u_h(x, y)-\bar{u}_{ij}\right) + \bar{u}_{ij} \quad \text{ where } 
\theta = \min\left\{1,  \Big| \frac{\bar{u}_{ij}-c_1}{\bar{u}_{ij} - m_1} \Big|, \frac{c_2-\bar{u}_{ij}}{m_2-\bar{u}_{ij}}\Big| \right\}, 
\end{equation}
 where 
 \begin{equation}\label{c12}
m_1 = \min\limits_{(x, y)\in S_{ij}} u_h(x, y)\quad \text{and}\quad  m_2=\max\limits_{(x, y)\in S_{ij}} u_h(x, y).
\end{equation}
The modified polynomials  are indeed in $[c_1, c_2]$ and preserve the cell average.  
Moreover, following \cite{LY14} it can be shown that  the above scaling limiters do not destroy the accuracy. We summarize this for two dimensional case only.  
\begin{lemma} If $\bar{u}_{ij} \in (c_1, c_2)$,  then the modified polynomial (\ref{recon2D}) is as accurate as $u_h(x,y)$ to approximate $u(t,x,y)$ for the same $t$ 
in the following sense:
$$
 |u_h(x, y)-\tilde u_h(x, y)| \leq C_k \|u_h(\cdot, \cdot) - u(t, \cdot, \cdot)\|_{\infty},
$$
where $C_k$ is a constant depending on the polynomial degree $k$ and the weight function $M(x,y)$.
\end{lemma}

\subsection{Algorithm}
The fact that we only require $u_h^n$ be in the desired range $[c_1, c_2]$ at certain points in $\cup_{i,j} S_{ij}$ can be used to reduce the computational cost in a great deal. 

Given the weighted $L^2$ projection $u_h^0$ computed from the initial data $u_0(x, y)$, the algorithm is stated below:
\begin{enumerate}
\item Initialization.

Obtain $u_h^0\in V_h$ using the standard piecewise $L^2$ projection 
$$
\int_{I_{ij}} u_0(x, y)\phi dx dy =\int_{I_{ij}} u_h^0\phi dx dy \quad {\rm for} \quad \phi \in V_h.
$$

\item Time evolution.\\
For $n = 0, 1, 2, ...$,
\begin{enumerate}
\item[(a)] Check the point values of $u_h^n$ on the test set $S_{ij}$. 
If one of them goes outside of $[c_1, c_2]$, reconstruct $\hat{u}_h^n$ using the formula (\ref{recon2D}) and (\ref{c12}) and set $u_h^n = \hat{u}_h^n$.

\item[(b)] Use the scheme (\ref{EFfully2D}) to compute $u_h^{n+1}$.
\end{enumerate}
End
\end{enumerate}

This algorithm is guaranteed to produce numerical solutions within the range with uniform third order accuracy for smooth exact solutions.

The algorithm with  forward Euler time discretization can be extended to high order ODE solves, such as the strong stability preserving Runge-Kutta methods, since they are a convex linear combination of the forward Euler; see \cite{ZS10}. 
The desired MPS property can be  ensured as long as the proper time step restrictions are respected.  

\section{Numerical tests}\label{secNum} 
In this section, we present the results of numerical tests using our third-order maximum-principle-preserving DDG schemes. 
The numerical integration is computed using the Gaussian quadrature rule. Since $M(x)$ and $A(x)$ can be complex functions, sufficient number of quadrature points will guarantee the desired order of accuracy and induce small numerical errors.  Hence we take $16$ quadrature points in each cell through all the examples.
As for the time stepping, we implement the SSP(3,3) scheme as in \cite{GKS09} for the strong-stability-preservation in time.
The one-dimensional error is measured by the discrete norms:
\begin{align*}
e^h_p(t) =\left(\sum_{j=1}^{N_x}\|u_h(t,\cdot) - u(t,\cdot)\|^p_{L^p(I_j)}\right)^{\frac1p}, \text{ for } p = 1 \text{ or } 2, 
\end{align*} 
and $u(t,x)$ is taken as the exact solution or the reference solution given by greatly refined spacial discretization. 
If neither of them is available, we can compute the consecutive errors between $u_h(t,x)$ and $u_{\frac h2}(t,x)$, where the subindex indicates the mesh size $h$ and $\frac h2$.
We introduce $e^h_\infty(t)$ to demonstrate the discrete MPS property at the test points in $S_j$: 
\begin{align}\label{einf}
e^h_\infty(t) = \max_{x \in S_j, 0 \leq j \leq N} \{c_1 - u_h(t, x), u_h(t, x) - c_2\},
\end{align}
If $e^h_\infty(t)>0$, then the discrete MPS property is violated.

\subsection{One dimensional numerical tests} 
In our numerical tests we choose scheme parameters as 
\[
\beta_0 = 2, \; \beta_1 = 0.16, \; \gamma = 0.1.
\]
\subsection*{Accuracy test}
We construct a linear problem of form \eqref{md} to demonstrate the third order accuracy of our numerical schemes:
\begin{equation}\label{Eg1D7}
\left\{\begin{array}{ll}
M(x)\partial_t u = \partial_x\big(A(x) \partial_xu\big), &\quad x \in [1, 3], \; t>0, \\
u(0, x)=\sin (x^2-1), &\quad x \in [1, 3]
\end{array}\right.
\end{equation}
with $M(x) = 4xe^{-x^2+1}, A(x) = \frac{e^{-x^2+1}}{x}$. The exact slution is given by 
$$
u(t,x) = \exp(-t)\sin(x^2-1-t).
$$
The boundary condition is imposed by using the exact solution.
We take the final time $t=0.1$ with different mesh sizes $N_x = 16, 32, 64$ and $128$. 
Fig. \ref{Eg1D7L2error} shows the logarithm of the $L^2$ error for $u_h$, denoted by circles.
One can observe the third order of accuracy for $u_h$.

\begin{figure}
\begin{minipage}{0.35\textwidth}
\includegraphics[width=\textwidth]{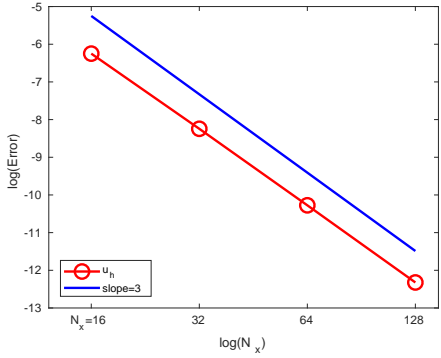}
\end{minipage}
\begin{minipage}{0.6\textwidth}
\caption{
The accuracy test on \eqref{Eg1D7}. 
The figure shows the logarithm of the $L^2$ error with different number of meshes for $u_h$, denoted by circles.
One can observe the third order of accuracy for $u_h$.}
\label{Eg1D7L2error}
\end{minipage}
\end{figure}

\subsection*{Porous medium equation} The porous medium equation
\begin{equation}\label{pm}
\partial_t u =\partial_x^2 (u^m), \quad m>1
\end{equation}
is known to admit the Barenblatt solution of the form 
\begin{equation*}
B_m(t, x)=t^{-\alpha}\left[ 1-\frac{\alpha(m-1)}{2m}\frac{|x|^2}{t^{2\alpha}} \right]_+^{\frac{1}{m-1}}, \quad \text{ with }\alpha = \frac{1}{m+1},
\end{equation*}
which is compactly supported.  
\begin{figure}
\begin{minipage}{0.35\textwidth}
\centering
\includegraphics[width=\textwidth]{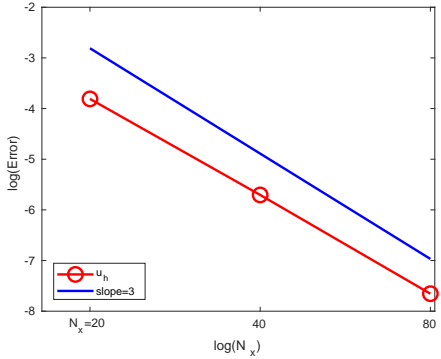}
\end{minipage}
\begin{minipage}{0.60\textwidth}
\caption{The accuracy test on \eqref{pm} with $m=5$ at $t = 0.1$. 
The error $e_p^h$ is computed by removing the two nonsmooth corners. 
The figure shows the logarithm of the $L^2$ error with different number of meshes for $u$, denoted by circles.
One can observe the third order of accuracy for $u_h$.}
\label{Eg1D2m5L2error}
\end{minipage}
\end{figure}
Fig. \ref{Eg1D2m5L2error} provides the accuracy test on \eqref{pm} with $m=5$ at $t = 0.1$.
Removing the two non-smooth corners, one can observe the third-order accuracy for the numerical solutions of $u_h$.

We compute the numerical solution with initial data $B_m(1,x)$ subject to zero boundary conditions for $m=2$, up to final time $t=3$. 
From the numerical results in Fig. \ref{Eg1D2}(a) with the MPS limiter we see a sharp resolution of discontinuities, and keeping the solution strictly within the initial bounds everywhere for all time.  
Fig. \ref{Eg1D2}(b) shows a zoom-in at the nonsmooth corner for $x \in [-6, -4]$ where the solution $u$ is well simulated.
In contrast, without MPS limiter, it brings in significant overshoots near the upper bound of the exact solution, as evidenced by oscillations already appearing at $t=1.0025$ in Fig. \ref{Eg1D2}(c).  In addition, the scheme without the MPS limiter will blow up in a short time.
\begin{figure}
\centering
\begin{subfigure}[b]{0.32\textwidth}
\includegraphics[width=\textwidth]{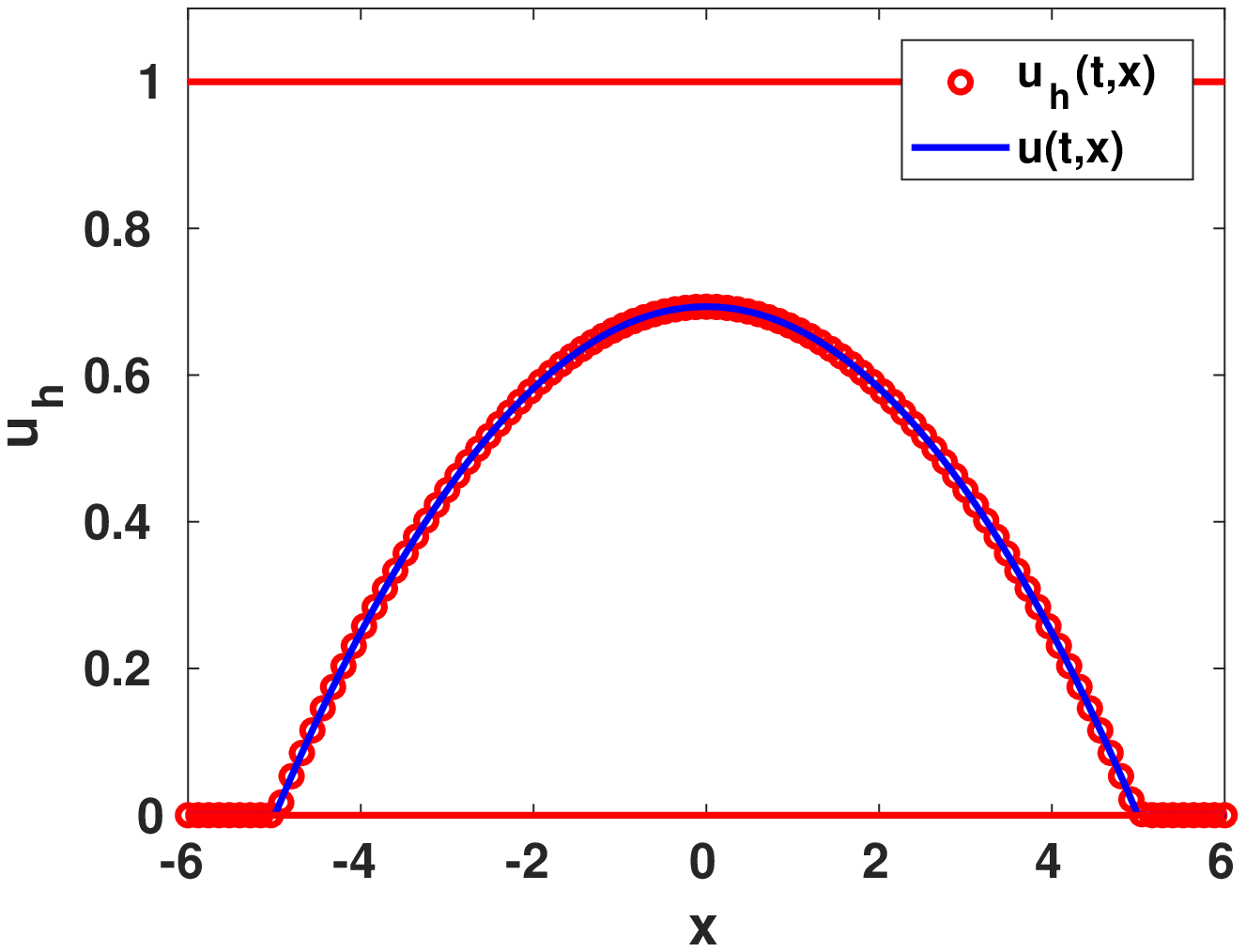}
\caption{$u_h(t=3,x)$.}
\label{Eg1D2uh_t3}
\end{subfigure}
\begin{subfigure}[b]{0.32\textwidth}
\includegraphics[width=\textwidth]{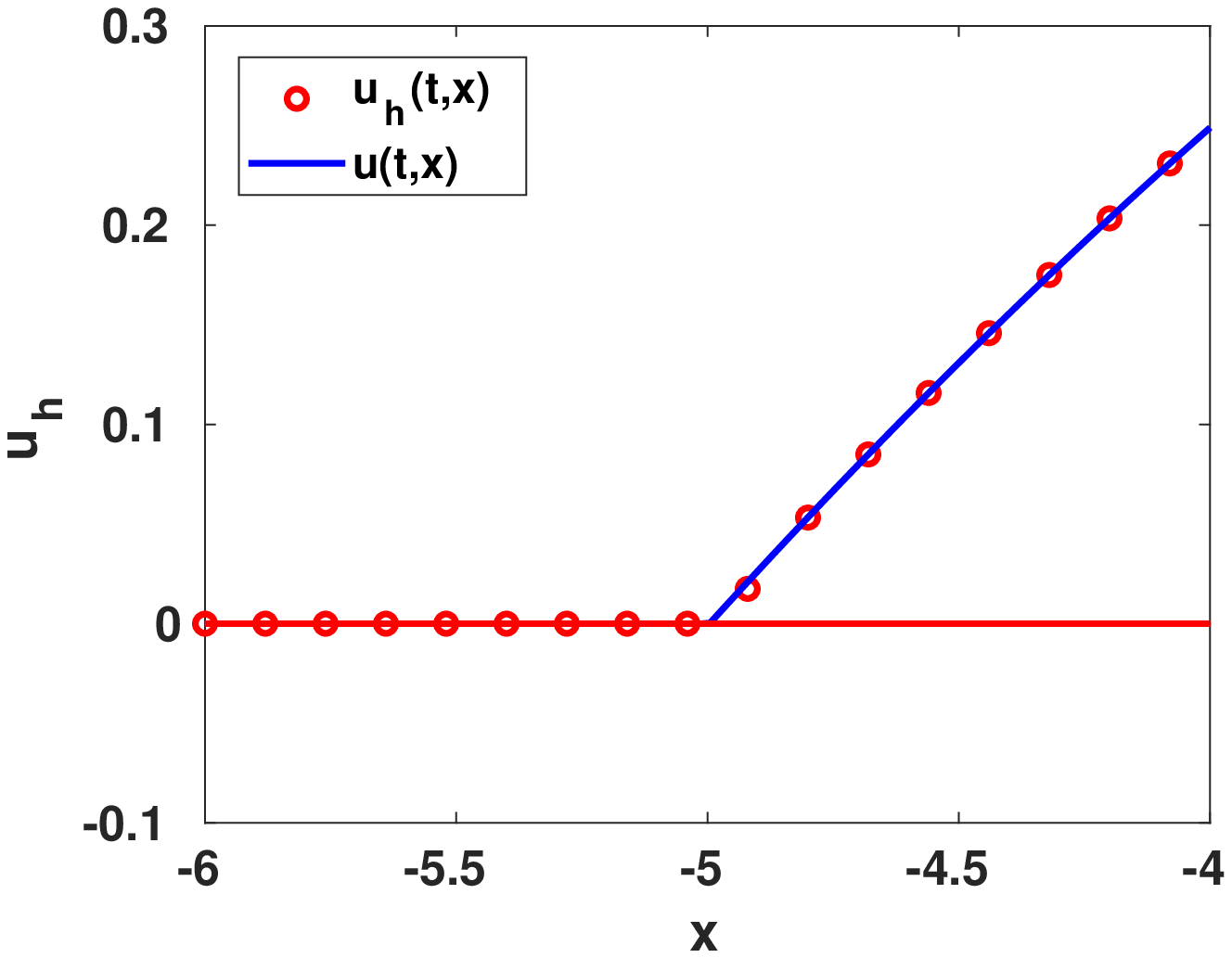}
\caption{Zoom-in of $u_h(t=3,x)$.}
\label{Eg1D2uh_t3}
\end{subfigure}
\begin{subfigure}[b]{0.32\textwidth}
\includegraphics[width=\textwidth]{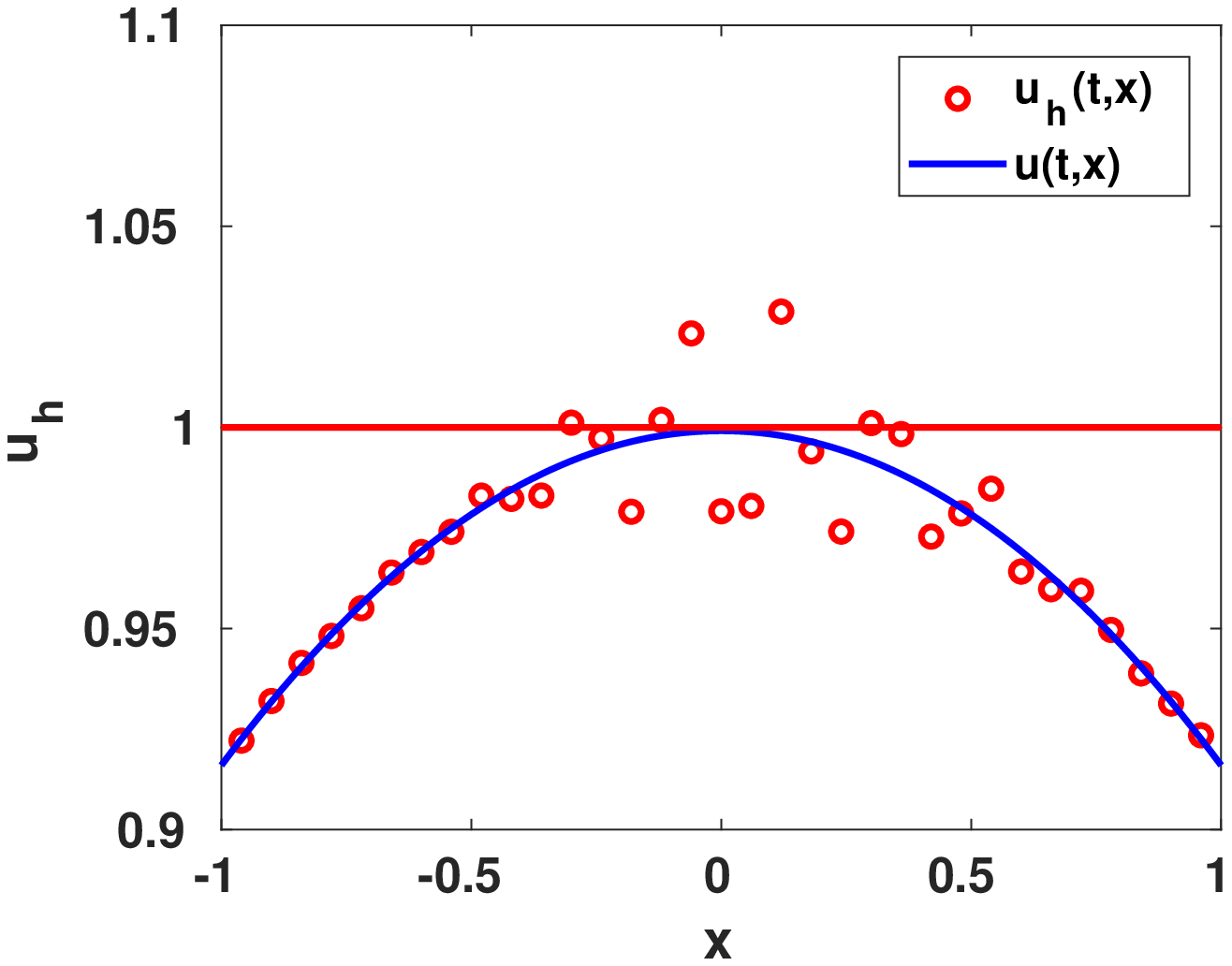}
\caption{$u_h(t=1.0025,x)$.}
\label{Eg1D2nolimiter}
\end{subfigure}
\caption{The numerical solution to \eqref{pm} with $m = 2$ and $N_x = 200, \Delta t = 0.0001$.
$\mu_0 \approx 3.66\times10^{-2}$.
Fig. \ref{Eg1D2uh_t3} shows $u_h(t=3,x)$ (circles) with the MPS limiter against the exact solution (solid lines).
Fig. \ref{Eg1D2nolimiter} shows the numerical solution without the MPS limiter at $t = 1.0025$, zoomed in  $[-1,1]$.
This numerical solution blows up shortly.
}
\label{Eg1D2}
\end{figure}
\subsection*{The Buckley-Leverett equation} The convection-diffusion Buckley-Leverett equation of the form 
\begin{equation}\label{bl}
\partial_t u +\partial_x f(u) =\varepsilon \partial_x (\nu(u) \partial_x u),
\end{equation}
is a model often used in reservoir simulations; see \cite{KT00}. Here $\varepsilon>0$ is a small parameter,  $f$ has an s-shape:
$$
f(u)=\frac{u^2}{u^2+(1-u)^2},
$$
and 
$$
\nu(u)=4 u(1-u)1_{0\leq u\leq 1}. 
$$
We numerically solve (\ref{bl}) with $\varepsilon=0.01$, subject to the following initial and boundary conditions 
\begin{equation}\label{bldata}
u(0, x)=(1-3x)1_{0\leq x\leq 1/3}, \quad u(t, 0)=1, \quad u(t, 1)=0.
\end{equation}
The exact solution is not available.  With numerical convergence we demonstrate that our numerical scheme is capable of simulating the sharp corner of the solution that is moving in time. From the results in Fig. \ref{Eg1D3uh} we observe the numerical convergence when the spacial mesh is refined. 
Moreover,  the lower bound  of $u_h(t,x)$ is well preserved around the corner point $x = 0.5$.
We note that the numerical solution here is comparable to that obtained in \cite{KT00} by the second order central scheme. 

\begin{figure}
\centering
\includegraphics[width=0.4\textwidth]{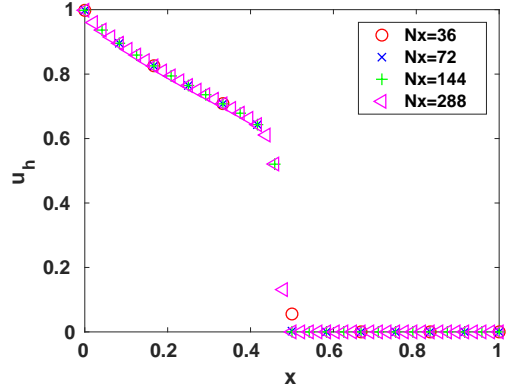}
\caption{The numerical solution of problem \eqref{bl} and \eqref{bldata} for $t = 0.2$ with 
$N_x = 36, 72, 144, 288$.}
\label{Eg1D3uh}
\end{figure}

\subsection{Two dimensional numerical tests}
We consider the convection-diffusion equation:
\begin{equation}\label{aniCD}
\partial_t u + \nabla\cdot f(u) = \nabla\cdot(A\nabla u) \text{ for } (x, y) \in [-1, 1]\times[-1,1] \text{ and } t>0,
\end{equation}
where
\begin{align*}
f(u) = u\vec v, \quad \vec v:=(0.01,0.01)^\top 
\end{align*}
and $A$ is a symmetric, positive definite matrix. 
For such an equation, exact solutions can be found of the form $u(t,x,y)=a(t)\exp(-\xi^\top B(t) \xi)$ with $\xi=\boldsymbol{x}-\vec vt$, $\boldsymbol{x}:= (x,y)^\top$, provided $a'=-a{\rm tr} (AB)$ and $B'=-2B^\top AB$. 
Here $a(0)$ can be chosen small enough to ensure that the periodic boundary condition adopted is reasonable and accurate at a finite time.  
In fact, if we set $\sigma_0 = 0.01$ and a $2\times2$ matrix $\sigma(t)$ be such that
\begin{equation*}
\sigma(t) = \sigma_0^2{\rm{Id}} + 2At \quad \text{ with } \rm{Id} \text{ being the identity matrix.}
\end{equation*}
Then the function
\begin{equation*}
u(t,x,y) = \frac{\sigma_0^2}{|\det(\sigma)|^\frac12}\exp\left(-\frac{(\boldsymbol{x}-\vec vt)^T\sigma^{-1}(\boldsymbol{x}-\vec vt)}{2}\right), \quad \boldsymbol{x}:= (x,y)^\top
\end{equation*}
 as given in \cite{CSG16}, is a solution to equation \eqref{aniCD}.

In our numerical tests,  we take three choices of the tensor $A$ as 
$$
A = 
\left(\begin{array}{cc}
1 & 0 \\ 0 & 1
\end{array}\right),
\left(\begin{array}{cc}
1 & 0 \\ 0 & 2
\end{array}\right)
\text{ or }
\left(\begin{array}{cc}
1 & 1 \\ 1 & 2
\end{array}\right),
$$
which are usually denoted as isotropic, diagonally anisotropic and fully anisotropic diffusion tensors.  

We begin to demonstrate the necessity of the MPS limiter using the first isotropic problem.
Table \ref{MPSlimiter} shows the minumum and maximum of the numerical solutions $u_h(t,x,y)$ using the MPS limiter over all the points in the test sets $S_{ij}$.
They are well bounded in the interval $[0,1]$, satisfying the maximum principle.
Table \ref{Nolimiter} shows the error $e_{\infty}^h(t)$ in \eqref{einf} without the MPS limiter.
One can observe that the numerical solutions $u_h(t,x,y)$ violates the maximum principle and the simulation will break down after some time.
Larger $\beta_0$ helps to suppress the overshoot or undershoot, but cannot realize the bound preservation ideally.
\begin{table}
\begin{minipage}[t]{0.45\textwidth}
\centering
\begin{tabular}{|c|c|c|}
\hline
$t$ & $\min(u)$& $\max(u)$ \\\hline
0 & 0 &1 \\\hline
$2\times10^{-6}$ & 0 & 0.96111\\\hline
$4\times10^{-6}$ & 0 & 0.92156\\\hline
$1\times10^{-5}$ & 0 & 0.81327\\\hline
$2\times10^{-5}$ & 0 & 0.68236\\\hline
\end{tabular}
\caption{MPS limiter}
\label{MPSlimiter}
\end{minipage}
\begin{minipage}[t]{0.45\textwidth}
\centering
\begin{tabular}{|c|c|c|}
\hline
$t$ & $\beta_0 = 2$& $\beta_0=4$ \\\hline
 0               & 1.705E-005 & 1.705E-005 \\\hline
$2\times10^{-6}$ & 7.023E-004 & 4.168E-004 \\\hline
$4\times10^{-6}$ & 1.339E-003 & 6.838E-004 \\\hline
$6\times10^{-6}$ & 1.874E-003 & 6.387E-004 \\\hline
$8\times10^{-6}$ & 2.403E-003 & 4.414E-004 \\\hline
\end{tabular}
\caption{No limiter: $e_{\infty}^h(t)$}
\label{Nolimiter}
\end{minipage}
\end{table}

In the first two cases, we take $\beta_0 = 2, \beta_1 = 0.16, \gamma = 0.1$ for both variables $x$ and $y$. 
For the fully anisotropic one, we take $\beta_0 = 4$. For smaller $\beta_0$,  small oscillations develop in time and the scheme can become unstable.
We observe a nice agreement of the numerical solution to the exact solution for all three cases of $A$. 

\begin{figure}
\centering
\begin{subfigure}[b]{0.45\textwidth}
\includegraphics[width=\textwidth]{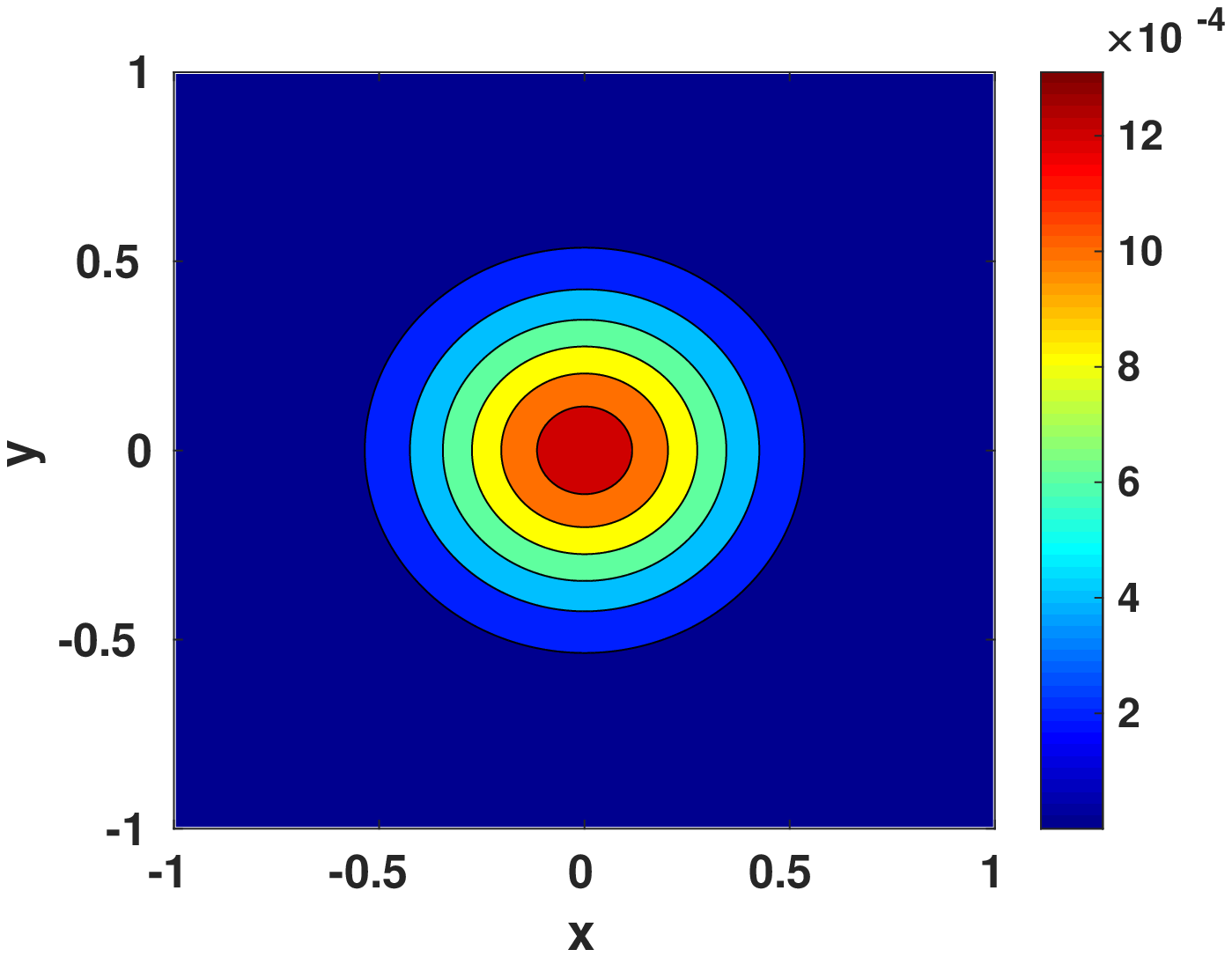}
\caption{$u(t,x)$.}
\label{Eg2DIso1ex}
\end{subfigure}
\begin{subfigure}[b]{0.45\textwidth}
\includegraphics[width=\textwidth]{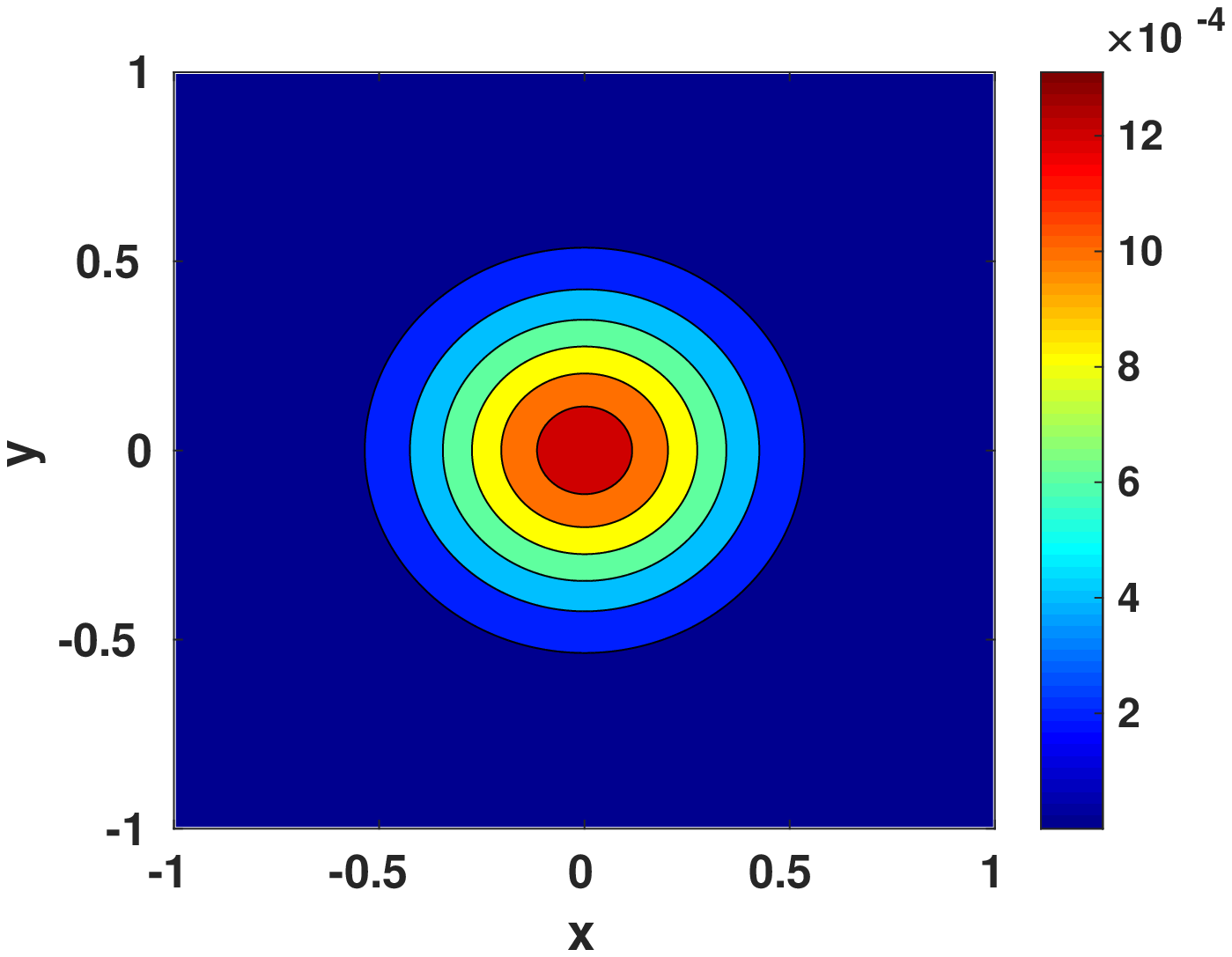}
\caption{$u_h(t,x)$.}
\label{Eg2DIso1}
\end{subfigure}
\caption{The contours of solutions to \eqref{aniCD} with the first choice of the matrix $A$ and $N_x = N_y = 200, \Delta t = 10^{-6}$.
$\mu_0 \approx 4.62\times10^{-3}$. 
The final time $t = 0.0381$.
Fig. \ref{Eg2DIso1ex} shows the exact solution.
Fig. \ref{Eg2DIso1} shows the numerical solution with the MPS limiter. 
}
\end{figure}

\begin{figure}
\centering
\begin{subfigure}[b]{0.45\textwidth}
\includegraphics[width=\textwidth]{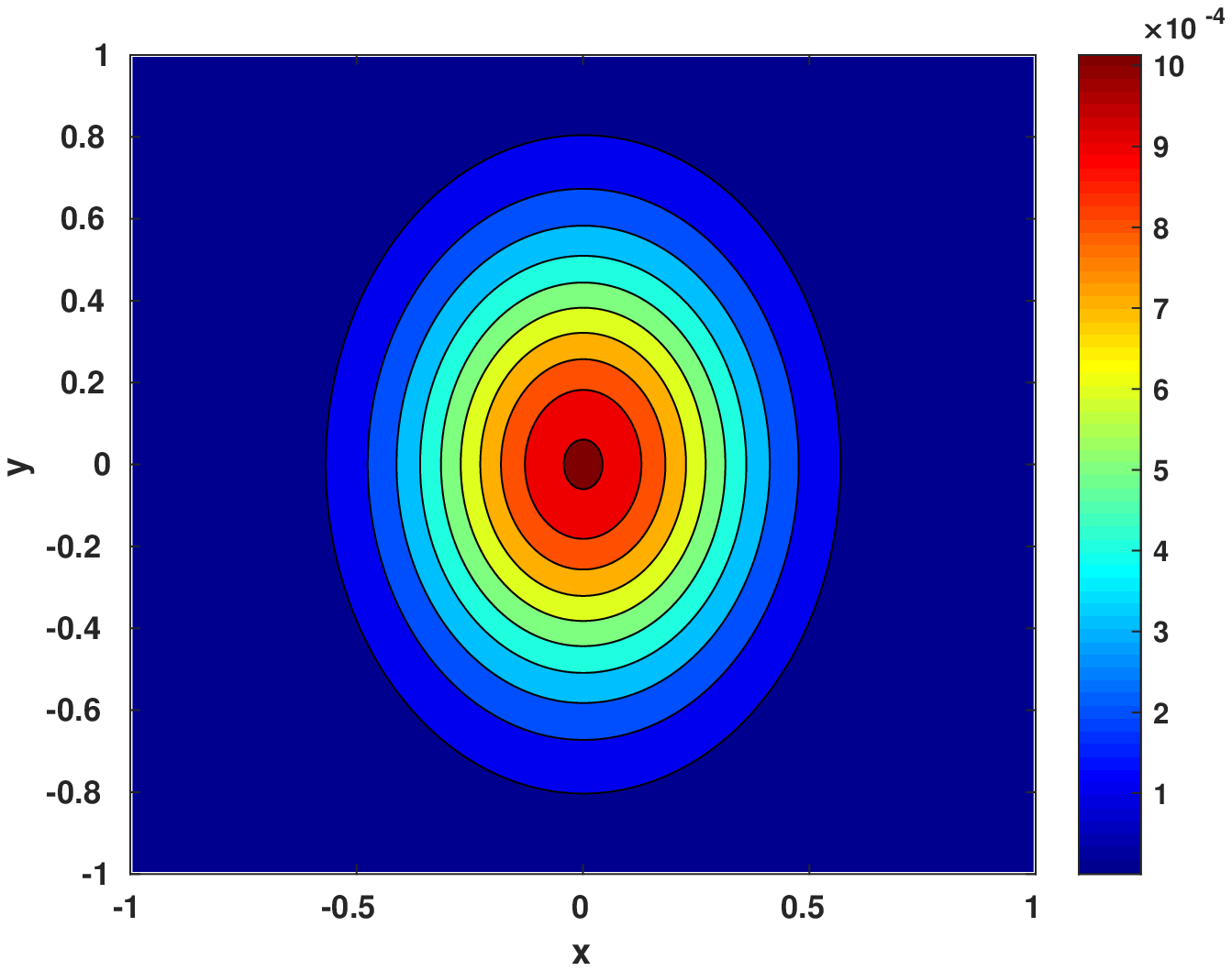}
\caption{$u(t,x, y)$.}
\label{Eg2DIso2ex}
\end{subfigure}
\begin{subfigure}[b]{0.45\textwidth}
\includegraphics[width=\textwidth]{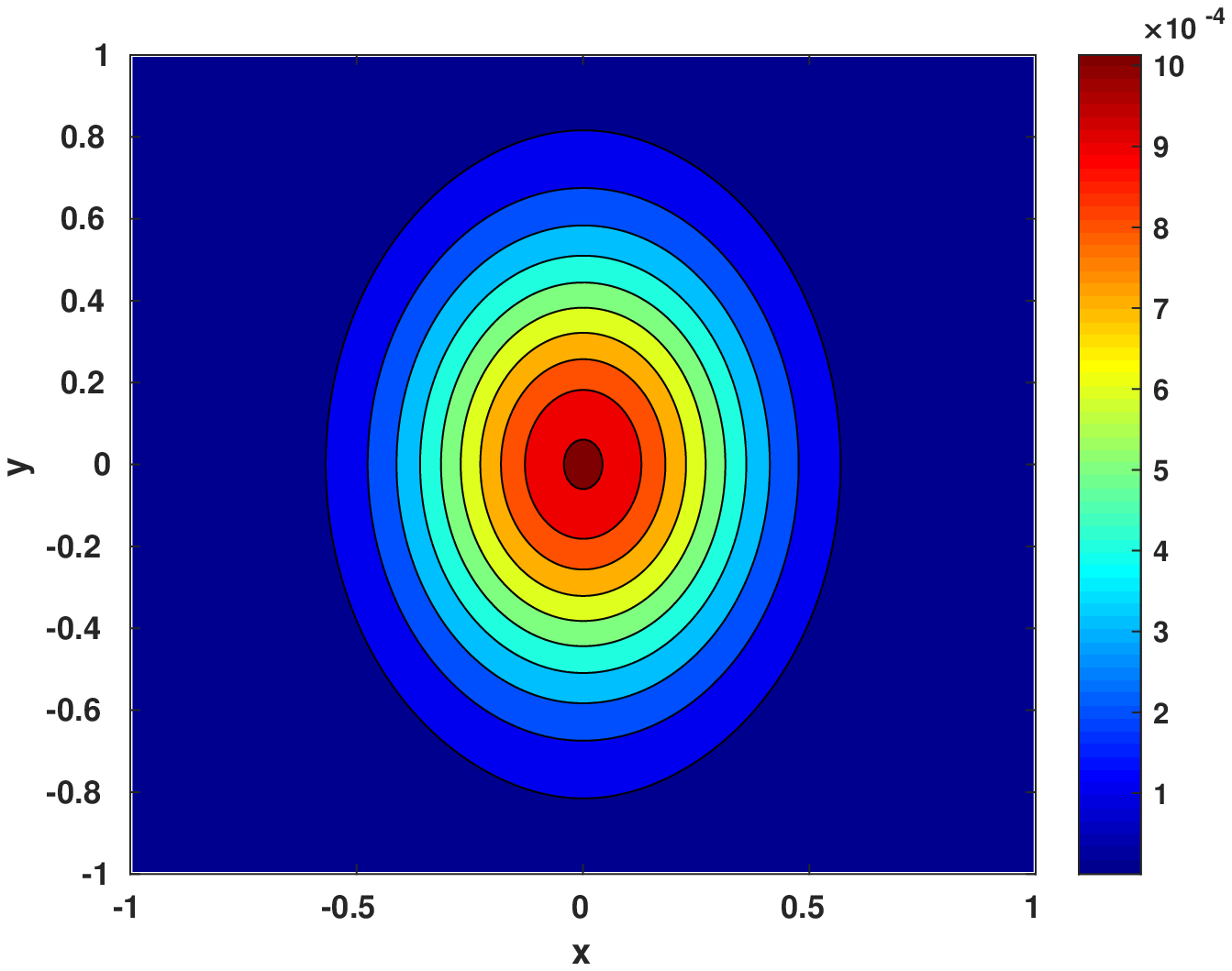}
\caption{$u_h(t,x, y)$.}
\label{Eg2DIso2}
\end{subfigure}
\caption{The contours of solutions to \eqref{aniCD} with the second choice of the matrix $A$ and $N_x = N_y = 200, \Delta t = 10^{-6}$. 
$\mu_0 \approx 4.62\times10^{-3}$.
The final time $t = 0.03485$.
Fig. \ref{Eg2DIso2ex} shows the exact solution.
Fig. \ref{Eg2DIso2} shows the numerical solution with the MPS limiter.
}
\end{figure}

For the equation \eqref{aniCD} with the third choice of the matrix $A$, we take care of the anisotropy by using a larger $\beta_0 = 4$, and adpative mesh size and time steps.
In the beginning, the solution is highly concentrated around the origin and requires a good resolution.
Therefore, for $t \in [0, 10^{-5}]$, we employ the discretization with $\Delta x = \Delta y = 0.005, \Delta t = 10^{-7}$.
Afterwards, a coarser mesh with $\Delta x = \Delta y = 0.01, \Delta t = 10^{-6}$ is used.
For time $t = 0.01$, Fig. \ref{Eg2DAni}(a-b) show the contours of the exact solution and the numerical solution with the MPS limiter.
Fig. \ref{Eg2DAni}(c) shows a slice of the exact solution (circles) and the numerical solution (crosses) for $y = 0$.
One can observe a good agreement of the two solutions.
\begin{figure}
\centering
\begin{subfigure}[b]{0.3\textwidth}
\includegraphics[width=\textwidth]{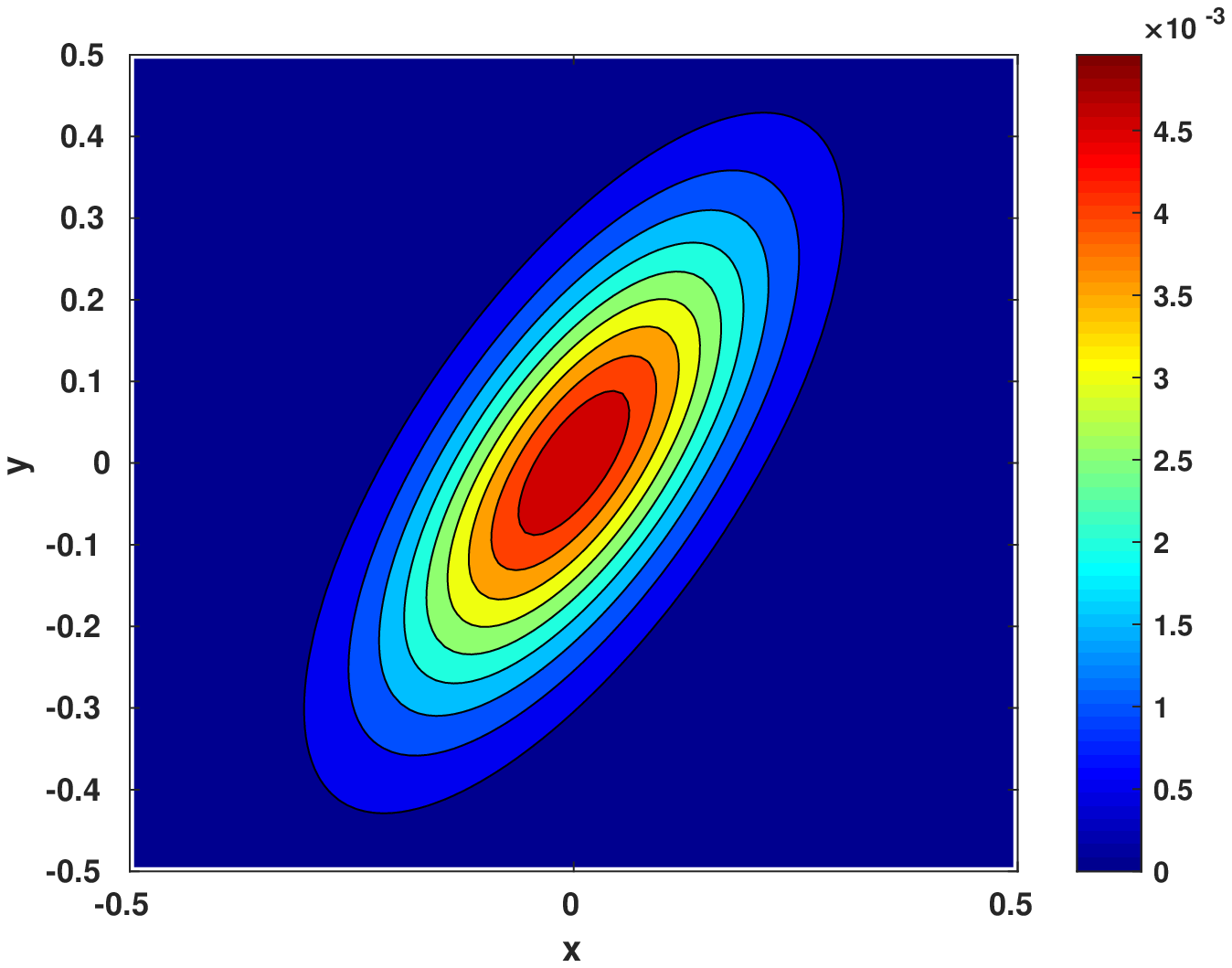}
\caption{$u(t,x,y)$.}
\label{Eg2DAniuex}
\end{subfigure}
\begin{subfigure}[b]{0.3\textwidth}
\includegraphics[width=\textwidth]{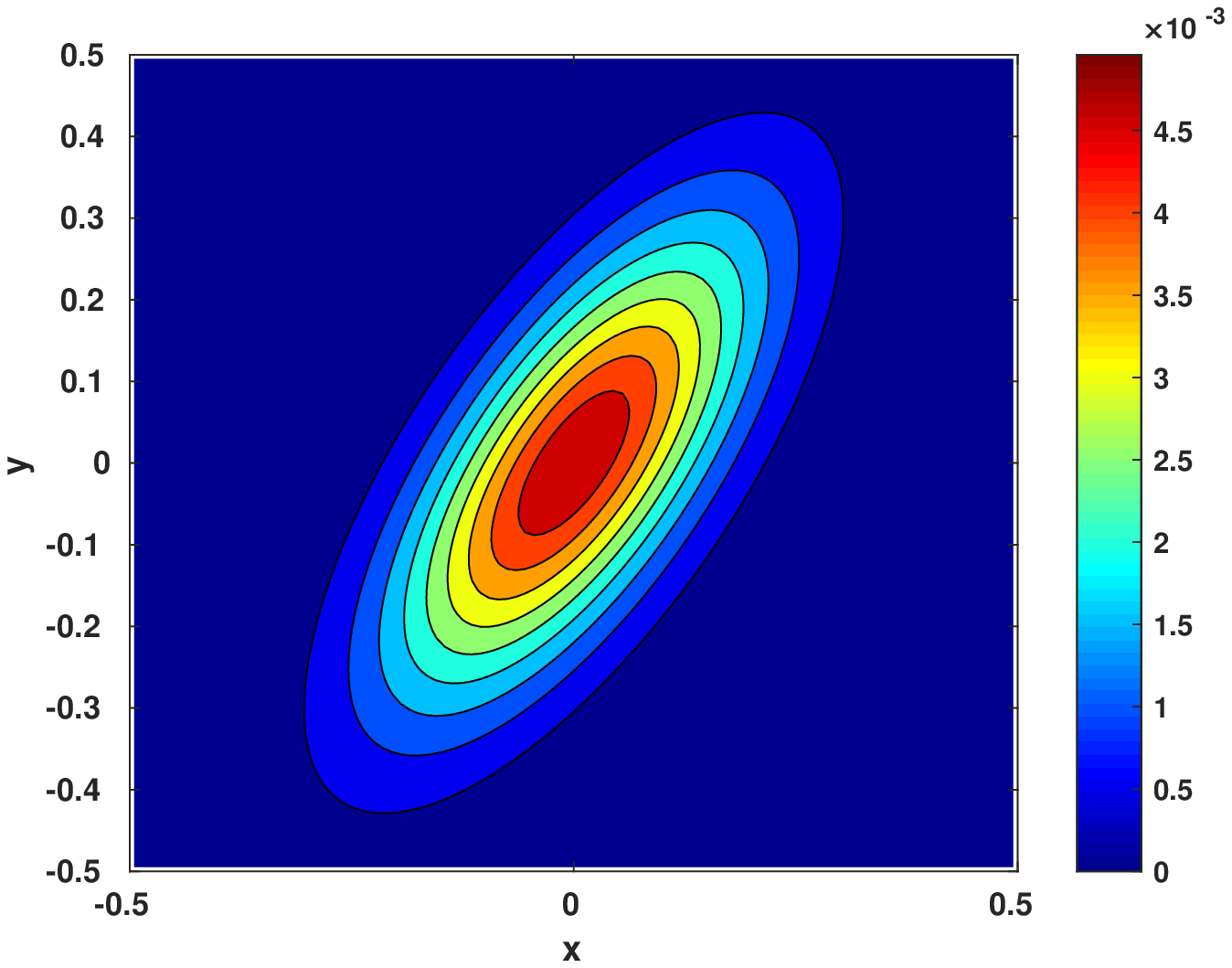}
\caption{$u_h(t,x,y)$.}
\label{Eg2DAniuh}
\end{subfigure}
\begin{subfigure}[b]{0.28\textwidth}
\includegraphics[width=\textwidth]{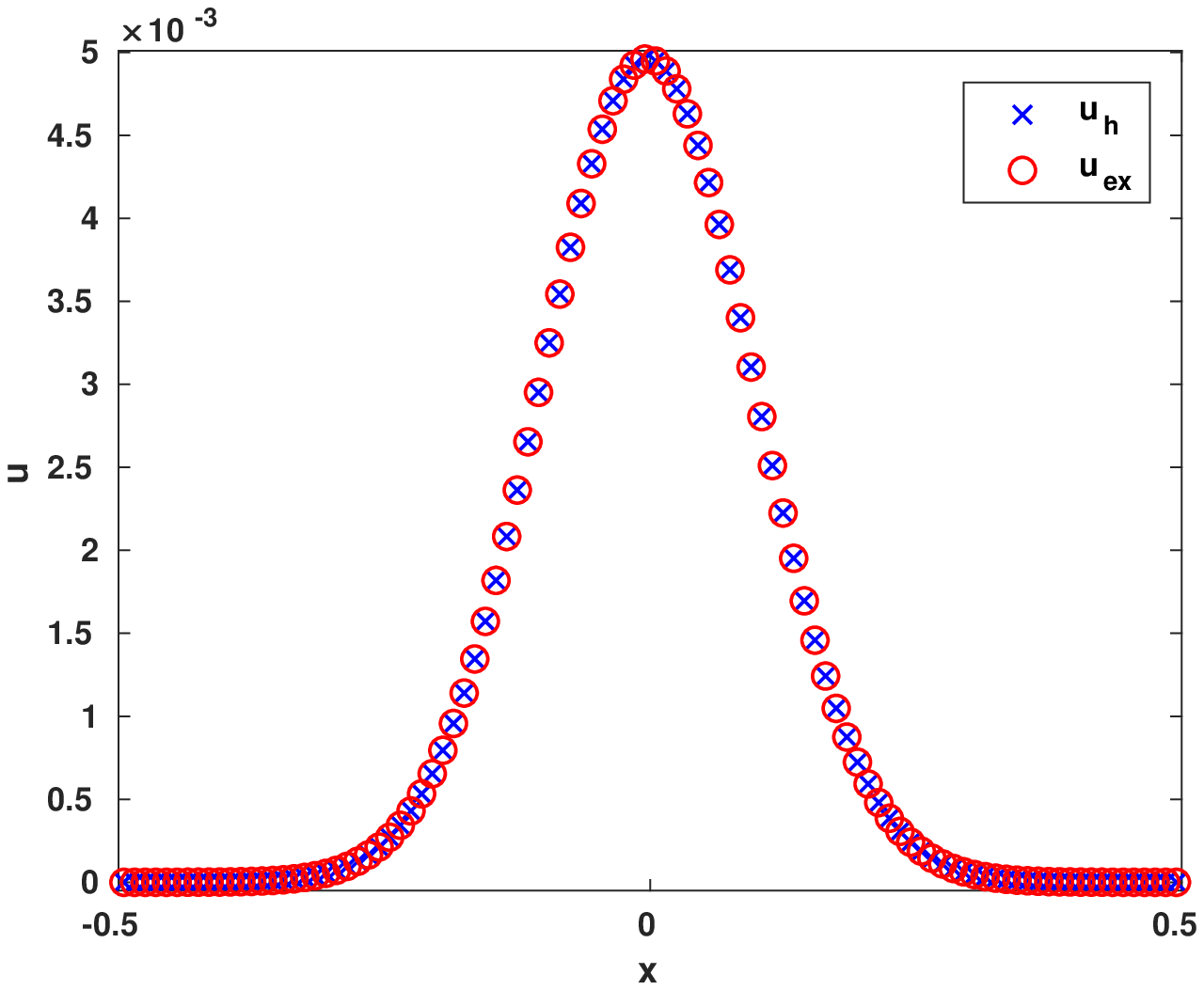}
\caption{$u_h(t,x,0)$.}
\label{Eg2DAniucut}
\end{subfigure}
\caption{
\eqref{aniCD} with the third choice of the matrix $A$ and adaptive mesh size and time steps.
$\mu_0 \approx 4.04\times10^{-3}$.
In the beginning, the solution is highly concentrated around the origin and requires a good resolution.
Therefore, for $t \in [0, 10^{-5}]$, we employ the discretization with $\Delta x = \Delta y = 0.005, \Delta t = 10^{-7}$.
Afterwards, a coarser mesh with $\Delta x = \Delta y = 0.01, \Delta t = 10^{-6}$ is used.
The final time $t = 0.01$.
Fig. \ref{Eg2DAniuex} shows the exact solution.
Fig. \ref{Eg2DAniuh} shows the numerical solution with the MPS limiter.
Fig. \ref{Eg2DAniucut} shows the good agreement between the exact solution (circles) and the numerical solution (crosses) for $y = 0$.
}
\label{Eg2DAni}
\end{figure}

\section{Concluding remarks}\label{secConclude}
In this paper, we present third order accurate DDG schemes which can be proven maximum-principle-satisfying for a class of diffusion equations with  variable diffusivity in terms of spatial variables and/or the unknown, in both one and two dimensional settings. 
Through careful theoretical analysis and numerical tests, we show that under suitable CFL conditions, with a simple scaling limiter involving little additional computational cost, the numerical schemes satisfy the strict maximum principle while maintaining uniform third order accuracy. 
The methodology extends to three dimensional rectangular meshes as well.
The effectiveness of the maximum-principle-satisfying DG schemes has been demonstrated through extensive numerical examples.

The presence of variable diffusivity has led to slightly more restrictive CFL conditions in order to preserve the maximum principle. 
In two dimensional case, the CFL condition seems a main factor for the slow numerical time evolution.  It would be interesting to extend the present result to implicit schemes so to improve the computational efficiency. 

\appendix 
\section{Proof of Theorem \ref{th2dk2}}
Upon regrouping, we decompose the right-hand side of (\ref{c5EF2DHQ}) as
\begin{align*}\label{EF2DHQ}
\langle u^{n+1}\rangle_{ij}
& =  \frac{\mu_x}{\mu}\sum_{\sigma=2}^{L-1} \hat\omega^{\sigma}H_1(\hat y_j^\sigma) + \frac{\mu_y}{\mu}\sum_{\sigma=2}^{L-1} \hat\omega^{\sigma}H_2(\hat x_i^\sigma) \\ \notag 
& \quad +\frac{\mu_x \hat \omega_1}{\mu} \big(H_1(y_{j-\frac12})+H_1(y_{j+\frac12}) \big) + \frac{\mu_y\hat \omega_1}{\mu} \big(H_2(x_{i-\frac12})+H_2(x_{i+\frac12}) \big) + B. \notag 
\end{align*}
Here $\hat\omega_1 = \hat\omega_L =\frac{1}{L(L-1)}$ is used. From \eqref{c5avedeck2} we see that 
$$
H_1(\hat y_j^\sigma)=R_{i}(M(\cdot,\hat y_j^\sigma), \mu, \Delta x, a), \quad H_2(\hat x_i^\sigma)=R_j(M(\hat x_i^\sigma, \cdot), \mu, \Delta y, b)
$$
for $1\leq \sigma \leq L$.
Notice that the terms in $B$, induced from the nontrivial $c$, involve only polynomials values at four vertices of $K_{ij}$, 
we proceed to regroup and combine them with $H_1(y_{j\pm\frac12})$ and $H_2(x_{i\pm\frac12})$, respectively, 
in the following  way: 
\begin{align*}
B_1 & = \frac{c\tau}{2\Delta x\Delta y}\big(u_h(x_{i-\frac12}^-, y_{j-\frac12}^+) + u_h(x_{i-\frac12}^+, y_{j-\frac12}^+) - u_h(x_{i+\frac12}^-, y_{j-\frac12}^+) - u_h(x_{i+\frac12}^+, y_{j-\frac12}^+)\big)  \; \text{with}\; H_1(y_{j-\frac12}),  \\
B_2 & = \frac{c\tau}{2\Delta x\Delta y}\big(u_h(x_{i+\frac12}^-, y_{j+\frac12}^-) - u_h(x_{i-\frac12}^+, y_{j+\frac12}^-) + u_h(x_{i+\frac12}^+, y_{j+\frac12}^-) - u_h(x_{i-\frac12}^-, y_{j+\frac12}^-)\big) \;  \text{with}\; H_1(y_{j+\frac12}), \\
B_3 & = \frac{c\tau}{2\Delta x\Delta y}\big(u_h(x_{i-\frac12}^+, y_{j-\frac12}^+) - u_h(x_{i-\frac12}^-, y_{j+\frac12}^-) + u_h(x_{i-\frac12}^+, y_{j-\frac12}^-) - u_h(x_{i-\frac12}^+, y_{j+\frac12}^+)\big) \;  \text{with}\;  H_2(x_{i-\frac12}), \\
B_4 & = \frac{c\tau}{2\Delta x\Delta y}\big(u_h(x_{i+\frac12}^-, y_{j-\frac12}^+) - u_h(x_{i+\frac12}^-, y_{j+\frac12}^-) + u_h(x_{i+\frac12}^-, y_{j-\frac12}^-) - u_h(x_{i+\frac12}^-, y_{j+\frac12}^+)\big) \;  \text{with}\;  H_2(x_{i+\frac12}). 
\end{align*}
We shall also use the following notations.  
\begin{align*}
\begin{array}{ll}
 \tilde\omega_{i}^{x,1}(\gamma^x,y) = \frac{\langle \gamma^x-\xi(1+\gamma^x) + \xi^2\rangle_i(y)}{2(1+\gamma^x)},
&\tilde\omega_{j}^{y,1}(\gamma^y,x) = \frac{\langle \gamma^y-\eta(1+\gamma^y) + \eta^2\rangle_j(x)}{2(1+\gamma^y)},\\
 \tilde\omega_{i}^{x,2}(\gamma^x,y) = \frac{\langle 1 - \xi^2\rangle_i(y)}{1-(\gamma^x)^2}, 
&\tilde\omega_{j}^{y,2}(\gamma^y,x) = \frac{\langle 1 - \eta^2\rangle_j(x)}{1-(\gamma^y)^2}, \\
 \tilde\omega_{i}^{x,3}(\gamma^x,y) = \frac{\langle -\gamma^x+\xi(1-\gamma^x) + \xi^2\rangle_i(y)}{2(1-\gamma^x)},
&\tilde\omega_{j}^{y,3}(\gamma^y,x) = \frac{\langle -\gamma^y+\eta(1-\gamma^y) + \eta^2\rangle_j(x)}{2(1-\gamma^y)}.
\end{array}
\end{align*}
For the first group, we have 
\begin{align*}
 &H_1(y_{j-\frac12}) + \frac{\mu}{\mu_x\hat\omega_1}B_1 = R_i(M(\cdot, y_{j-\frac12}), \mu, \Delta x, a) + \frac{\mu}{\mu_x\hat\omega_1}B_1 \notag\\
=& \left[\tilde{\omega}_{i}^{x,1}(\gamma^x, y_{j-\frac12}) - \mu a \big(\alpha_3(-\gamma^x) +\alpha_1(\gamma^x)\big) + \frac{c\tau\mu}{2\Delta x\Delta y\mu_x\hat\omega_1} \right]u_h(x_{i-\frac12}^+, y_{j-\frac12}^+) \nonumber\\
 &+\left[\tilde{\omega}_{i}^{x,2}(\gamma^x, y_{j-\frac12}) - \mu a \big(\alpha_2(-\gamma^x) +\alpha_2(\gamma^x)\big)\right]u_h(x_i^\gamma,y_{j-\frac12}^+)\nonumber\\
 &+\left[\tilde{\omega}_{i}^{x,3}(\gamma^x, y_{j-\frac12}) - \mu a \big(\alpha_1(-\gamma^x) +\alpha_3(\gamma^x)\big) - \frac{c\tau\mu}{2\Delta x\Delta y\mu_x\hat\omega_1}\right]u_h(x_{i+\frac12}^-, y_{j-\frac12}^+) \nonumber\\
 &+\left(\mu a\alpha_3(-\gamma^x) - \frac{c\tau\mu}{2\Delta x\Delta y\mu_x\hat\omega_1}\right)u_h(x_{i+\frac12}^+, y_{j-\frac12}^+)\nonumber \\
&   +\mu a\left[\alpha_2(-\gamma^x)u_h(x_{i+1}^\gamma,y_{j-\frac12}^+)+\alpha_1(-\gamma^x)u_h(x_{i+\frac32}^-, y_{j-\frac12}^+) \right]\nonumber \\
 &+\mu a\left[\alpha_1(\gamma^x) u_h(x_{i-\frac32}^+, y_{j-\frac12}^+)+\alpha_2(\gamma^x)u_h(x_{i-1}^\gamma,y_{j-\frac12}^+) \right]\nonumber \\ 
 & + \left(\mu a\alpha_3(\gamma^x) + \frac{c\tau\mu}{2\Delta x\Delta y\mu_x\hat\omega_1}\right)u_h(x_{i-\frac12}^-, y_{j-\frac12}^+),
\end{align*}
and the second group reduces to  
\begin{align*}
 &H_1(y_{j+\frac12}) + \frac{\mu}{\mu_x\hat\omega_1}B_2 = R_i(M(\cdot, y_{j+\frac12}), \mu, \Delta x, a) + \frac{\mu}{\mu_x\hat\omega_1}B_1 \notag\\
=& \left[\tilde{\omega}_{i}^{x,1}(\gamma^x, y_{j+\frac12}) - \mu a \big(\alpha_3(-\gamma^x) +\alpha_1(\gamma^x)\big) - \frac{c\tau\mu}{2\Delta x\Delta y\mu_x\hat\omega_1} \right]u_h(x_{i-\frac12}^+, y_{j+\frac12}^-) \nonumber\\
 &+\left[\tilde{\omega}_{i}^{x,2}(\gamma^x, y_{j+\frac12}) - \mu a \big(\alpha_2(-\gamma^x) +\alpha_2(\gamma^x)\big)\right]u_h(x_i^\gamma,y_{j+\frac12}^-)\nonumber\\
 &+\left[\tilde{\omega}_{i}^{x,3}(\gamma^x, y_{j+\frac12}) - \mu a \big(\alpha_1(-\gamma^x) +\alpha_3(\gamma^x)\big) + \frac{c\tau\mu}{2\Delta x\Delta y\mu_x\hat\omega_1}\right]u_h(x_{i+\frac12}^-, y_{j+\frac12}^-) \nonumber\\
 &+\left(\mu a\alpha_3(-\gamma^x) + \frac{c\tau\mu}{2\Delta x\Delta y\mu_x\hat\omega_1}\right)u_h(x_{i+\frac12}^+, y_{j+\frac12}^-) \nonumber \\
  & +\mu a\left[\alpha_2(-\gamma^x)u_h(x_{i+1}^\gamma,y_{j+\frac12}^-)+\alpha_1(-\gamma^x)u_h(x_{i+\frac32}^-, y_{j+\frac12}^-) \right]\nonumber \\
 &+\mu a\left[\alpha_1(\gamma^x) u_h(x_{i-\frac32}^+, y_{j+\frac12}^-)+\alpha_2(\gamma^x)u_h(x_{i-1}^\gamma,y_{j+\frac12}^-) \right]\nonumber \\
 & +
   \left(\mu a\alpha_3(\gamma^x) - \frac{c\tau\mu}{2\Delta x\Delta y\mu_x\hat\omega_1}\right)u_h(x_{i-\frac12}^-, y_{j-\frac12}^+). \nonumber
\end{align*}
From the above two groups we see that all coefficients of solution values involved  are nonnegative if 
\begin{subequations}\label{cfl1}
\begin{align}
& a\alpha_3(\pm\gamma^x) - \frac{\Delta x|c|}{2\Delta y\hat\omega_1} \geq 0, \\
& \mu \leq \min\left\{\frac{\tilde\omega_{i}^{x,1}(\pm\gamma^x, y_{j\pm\frac12})}{a\big(\alpha_1(\pm\gamma^x) + \alpha_3(\mp\gamma^x)\big) + \frac{\Delta x |c|}{\Delta y\hat\omega_1}},
\frac{\tilde{\omega}_{i}^{x,2}(\gamma^x, y_{j\pm \frac12})}{2\alpha_2(\gamma^x)}\right\}.
\end{align}
\end{subequations}
Here we used $\tilde{\omega}_{i}^{x,3}(\gamma^x, y) = \tilde{\omega}_{i}^{x,1}(-\gamma^x, y)$ and $\alpha_2(\gamma^x) = \alpha_2(-\gamma^x)$.

Similarly,  all coefficients of solution values in 
\begin{align*}
H_2(x_{i-\frac12}) + \frac{\mu}{\mu_y\hat\omega_1}B_3 = R_j(M(x_{i-\frac12}, \cdot), \mu, \Delta y, b) + \frac{\mu}{\mu_y\hat\omega_1}B_3 
\end{align*}
and 
\begin{align*}
H_2(x_{i+\frac12}) + \frac{\mu}{\mu_y\hat\omega_1}B_4 = R_j(M(x_{i+\frac12}, \cdot), \mu, \Delta y, b) + \frac{\mu}{\mu_y\hat\omega_1}B_4
\end{align*}
are also nonnegative if
\begin{subequations}\label{cfl2}
\begin{align}
& b\alpha_3(\pm\gamma^y) - \frac{\Delta y|c|}{2\Delta x\hat\omega_1} \geq 0,\\
& \mu \leq \min\left\{\frac{\tilde\omega_{j}^{y,1}(\pm\gamma^y, x_{i\pm\frac12})}{b\big(\alpha_1(\pm\gamma^y) + \alpha_3(\mp\gamma^y)\big) + \frac{\Delta y |c|}{\Delta x\hat\omega_1}},
\frac{\tilde\omega_{j}^{y,2}(\gamma^y, x_{i\pm\frac12})}{2\alpha_2(\gamma^y)}\right\}.
\end{align}
\end{subequations}
Here we used $\tilde{\omega}_{j}^{y,3}(\gamma^y, x) = \tilde{\omega}_{j}^{y,1}(-\gamma^y,x)$ and $\alpha_2(\gamma^y) = \alpha_2(-\gamma^y)$.

Since $\tilde\omega_i^{x,\sigma}, \tilde\omega_j^{y,\sigma}$ for $\sigma = 1, 2, 3$ only depend on $M(x,y)|_{K_{ij}}$ and $\gamma$, therefore bounded  from below. We set such bound as 
\begin{equation}\label{min_omega}
\underline\omega_{ij} = \min_{\gamma^x, \gamma^y, x\in \hat S_{i}^x, y\in \hat S_{j}^y}\{\tilde\omega_{i}^{x,1}(\pm\gamma^x, y), 
\tilde\omega_{i}^{x,2}(\gamma^x,y), 
\tilde\omega_{j}^{y,1}(\pm\gamma^y, x), 
\tilde\omega_{j}^{y,2}(\gamma^y, x)\}. 
\end{equation}
Notice also that for $\gamma = \max\{|\gamma^x|, |\gamma^y|\}$,  using $\beta_1\leq 1/4$, we have 
\[
\alpha_1(\pm\gamma^x) + \alpha_3(\mp\gamma^x) = \beta_0 + \frac{8\beta_1-2}{1\pm\gamma^x} \leq \beta_0 + \frac{8\beta_1-2}{1-\gamma},
\]
which also hods when $\gamma^x$ is replaced by $\gamma^y$.  Hence both (\ref{cfl1}b) and  (\ref{cfl2}b) are ensured  by  \eqref{CFL2d-}.

Observe that (\ref{cfl1}a) and  (\ref{cfl2}a) are implied by 
$$
\beta_0 +\min_{s\in [-\gamma, \gamma]}\frac{8\beta_1-3+s}{2(1-s)} \geq \frac{|c|\kappa}{2\hat \omega_1 \min\{a, b\}}.
$$
Using $\beta_1 \leq 1/4$ we see that  the minimum on the left hand side is $-1$, obtained at $s=\gamma$ when $\gamma = 8\beta_1-1$, hence this relation gives the lower bound in (\ref{beta2dp}). 

\section{Proof of Theorem \ref{th2dk2+} } 
For simplicity of presentation, in the following we consider  $A=A(x, y)$ instead of $A(x, y, u)$, for which we only need to use $\{A_h\}$ on interfaces, so that $c=c(x, y)$, $a=a(x, y)$ and $b=b(x, y)$. 
 
Since $u_h(x,y)$ is quadratic in terms of $x$ and $y$ respectively,  we can use the formula (\ref{pj}) to obtain 
\begin{align*}
\dot p(\eta)=&\dot\omega^1(\eta)p(-1) + \dot\omega^2(\eta)(\gamma) + \dot\omega^3(\eta)p(1) \\
=& \frac{2\eta -1-\gamma}{2(1+\gamma)}{p}(-1)+\frac{2\eta }{(\gamma^2 -1)}{p}(\gamma)+\frac{2\eta +1-\gamma}{2(1-\gamma)}{p}(1).
\end{align*}
Here $\eta$ is the variable in the reference element $[-1,1]$, 
and $\dot\omega^\sigma(\eta) = \frac{d}{d\eta}\omega^\sigma(\eta)$.
Then 
\begin{align*}
\int_{J_j} c(x,y)\{\partial_y  u_h(x, y)\}dy &= \{u_h(x, y_{j-\frac12}^+)\} \int_{-1}^1 c\left(x,y_j+\frac{\Delta y}{2}\eta \right)\dot\omega^1(\eta)\,d\eta  \\
 &\quad +\{u_h(x, y^\gamma_{j})\} \int_{-1}^1 c\left(x, y_j+\frac{\Delta y}{2}\eta\right)\dot\omega^2(\eta)\,d\eta \\
 &\quad +\{u_h(x, y_{j+\frac12}^-)\} \int_{-1}^1 c\left(x, y_j+\frac{\Delta y}{2}\eta\right)\dot\omega^3(\eta)\,d\eta.
\end{align*}
Here the average $\{\cdot\}$ is taken with respect to $x^-$ and $x^+$.
Using the quadrature rule on the right-hand side of (\ref{EF2Dv1+}), we obtain the following scheme
\begin{equation}\label{c5EF2DHQ+}
\bar u^{n+1}_{ij}= \frac{\mu_x}{\mu}\sum_{\sigma=1}^L \hat\omega^{\sigma}H_1(\hat y_j^\sigma) + \frac{\mu_y}{\mu}\sum_{\sigma=1}^L \hat\omega^{\sigma}H_2(\hat x_i^\sigma)+B. 
\end{equation}
The test set now reduces to 
\begin{align*}
& S^x_i = \{x_{i-\frac12}, x_i^{\gamma}, x_{i+\frac12}\} = x_i + \frac{\Delta x}{2}\{-1, \gamma, 1\},  \\
& S^y_j = \{y_{j-\frac12}, y_j^{\gamma}, y_{j+\frac12}\} = y_j + \frac{\Delta y}{2}\{-1, \gamma, 1\}
\end{align*}
with $\gamma$ satisfying
\begin{align} \label{c5angle2D+}
&|\gamma| \leq \frac13  \; \text{and} \; |\gamma|\leq 8\beta_1-1.
\end{align}
It can be shown that there exists $L \geq \frac{2k+3}{2}$ and $\gamma$ satisfying  (\ref{c5angle2D+})  such that $x_i^\gamma \in \hat S^x_i$ and $y_j^\gamma\in \hat S^y_j$, 
the corresponding quadrature weight is denoted by $\hat \omega^*$.
Upon regrouping,  the right-hand side of (\ref{c5EF2DHQ+}) may be decomposed  as 
\begin{align*}
\bar u^{n+1}_{ij}
 = &\quad \frac{\mu_x}{\mu}\sum_{\sigma\neq 1, *, L} \hat\omega^{\sigma}H_1(\hat y_j^\sigma) + \frac{\mu_y}{\mu}\sum_{\sigma\neq1,*,L} \hat\omega^{\sigma}H_2(\hat x_i^\sigma) \\
& + \frac{\mu_x}{\mu} \big(\hat \omega^1 H_1(y_{j-\frac12})+ \hat \omega^* H_1(y_j^\gamma)+ \hat \omega^1 H_1(y_{j+\frac12}) \big) \\
& + \frac{\mu_y}{\mu} \big(\hat \omega^1 H_2(x_{i-\frac12})+ \hat \omega^* H_2(x_i^\gamma)+ \hat \omega^1 H_2(x_{i+\frac12}) \big) + B. 
\end{align*}
Here $\hat\omega^1 = \hat\omega^L$ is used. 
The terms in $B$ will be regrouped correspondingly.  More precisely, the first term involving integration on $J_j$ is regrouped in terms involving solution values at $y_{j-\frac12}^+, y_j^\gamma$ and $y_{j+\frac12}^-$,  and combined with $H_1(y_{j-\frac12}), H_1(y_j^\gamma)$ and $H_1(y_{j+\frac12})$, respectively.
We check upon the first group only: 
\begin{align*}
 &H_1(y_{j-\frac12}) + \frac{\mu}{\mu_x\hat\omega^1}B_1 = R_i(M(\cdot, y_{j-\frac12}), \mu, \Delta x, a) + \frac{\mu}{\mu_x\hat\omega^1}B_1 \\
=&\left[\omega^1 - \mu a \big(\alpha_3(-\gamma) +\alpha_1(\gamma)\big) - \frac{\mu\Delta x}{2\Delta y\hat\omega^1} 
   \int_{-1}^1c\left(x_{i-\frac12}, y_j+\frac{\Delta y}{2}\eta\right)\dot\omega^1(\eta)\,d\eta \right]u_h(x_{i-\frac12}^+, y_{j-\frac12}^+) \\
 &+\left[\omega^2 - \mu a \big(\alpha_2(-\gamma) +\alpha_2(\gamma)\big)\right]u_h(x_i^\gamma,y_{j-\frac12}^+)\\
 &+\left[\omega^3 - \mu a \big(\alpha_1(-\gamma) +\alpha_3(\gamma)\big) +\frac{\mu\Delta x}{2\Delta y\hat\omega^1}
   \int_{-1}^1c\left(x_{i+\frac12}, y_j+\frac{\Delta y}{2}\eta\right)\dot\omega^1(\eta)\,d\eta \right]u_h(x_{i+\frac12}^-, y_{j-\frac12}^+) \\
 &+\left(\mu a\alpha_3(-\gamma) +\frac{\mu\Delta x}{2\Delta y\hat\omega^1}\int_{-1}^1c\left(x_{i+\frac12}, y_j+\frac{\Delta y}{2}\eta\right)\dot\omega^1(\eta)\,d\eta
   \right)u_h(x_{i+\frac12}^+, y_{j-\frac12}^+) \\
 &+\mu a\left[\alpha_2(-\gamma)u_h(x_{i+1}^\gamma, y_{j-\frac12}^+)+\alpha_1(-\gamma)u_h(x_{i+\frac32}^-,y_{j-\frac12}^+) \right] \\
 &+\mu a\left[\alpha_1(\gamma) u_h(x_{i-\frac32}^+, y_{j-\frac12}^+)+\alpha_2(\gamma)u_h(x_{i-1}^\gamma, y_{j-\frac12}^+) \right] \\ 
 &+\left(\mu a\alpha_3(\gamma) -\frac{\mu\Delta x}{2\Delta y\hat\omega^1}\int_{-1}^1c\left(x_{i-\frac12}, y_j+\frac{\Delta y}{2}\eta\right)\dot\omega^1(\eta)\,d\eta \right)
   u_h(x_{i-\frac12}^-, y_{j-\frac12}^+),
\end{align*}
From the three groups of $H_1(y_{j-\frac12}), H_1(y_j^\gamma)$ and $H_1(y_{j+\frac12})$, we see that all the coefficients of solutions values are nonnegative if 
\begin{align*}
& a\alpha_3(\pm\gamma) - \frac{\Delta x}{2\Delta y \min\{\hat\omega_1, \hat \omega^*\}}
\max_{x \in S_i^x; \sigma}
\left|g_j^\sigma(x) \right| \geq 0,\\ 
& \mu \leq \min\left\{\frac{\min\{\omega^1,\omega^3\}}{a\big(\alpha_1(\pm\gamma) + \alpha_3(\mp\gamma)\big) + \frac{\Delta x}{2\Delta y \min\{\hat\omega_1, \hat \omega^*\} }\max_{x \in S_i^x; \sigma} \left|g_j^\sigma(x) \right|},
\frac{\omega^2}{2a\alpha_2(\gamma)}\right\},
\end{align*}
where $g_j^\sigma(x)=\int_{-1}^1 c\left(x, y_j+\frac{\Delta y}{2}\eta\right)\dot\omega^\sigma(\eta)\,d\eta$.  The rest of terms in $B$ when combined with $H_2(y)$ for $y\in S_j^y$ led to the following conditions 
\begin{align*}
& b \alpha_3(\pm\gamma) - \frac{\Delta y}{2\Delta x \min\{\hat\omega_1, \hat \omega^*\}}
\max_{y \in S_j^y; \sigma}
\left|\int_{-1}^1 c\left(x_i+\frac{\Delta x}{2}\xi, y\right)\dot\omega^\sigma(\xi)\,d\xi \right| \geq 0,\\ 
& \mu \leq \min\left\{\frac{\min\{\omega^1,\omega^3\}}{b\big(\alpha_1(\pm\gamma) + \alpha_3(\mp\gamma)\big) + \frac{\Delta y}{2\Delta x \min\{\hat\omega_1, \hat \omega^*\} }\max_{y \in S_j^y; \sigma} \left|g_i^\sigma(y)\right|},
\frac{\omega^2}{2a\alpha_2(\gamma)}\right\},
\end{align*}
where $g_i^\sigma(y)=\int_{-1}^1 c\left(x_i+\frac{\Delta x}{2}\xi, y)\right)\dot\omega^\sigma(\xi)\,d\xi$. 
Note that both $\left|\int_{-1}^1c\left(x, y_j+\frac{\Delta y}{2}\eta\right)\dot\omega^\sigma(\eta)\,d\eta\right|$ and $\left|\int_{-1}^1c\left(x_i+\frac{\Delta x}{2}\xi, y)\right)\dot\omega^\sigma(\xi)\,d\xi\right|$ are bounded from above by 
\begin{align*}
\|c\|_\infty \left|\int_{-1}^1\dot\omega^\sigma(\eta) d\eta \right| \leq \frac{4\|c\|_\infty}{1-\gamma}.
\end{align*}
Using the fact that $\kappa^{-1} \leq \frac{\Delta x}{\Delta y}\leq \kappa$, $\omega^3 \leq \omega^1$, $\frac{1}{L(L-1)}=\hat \omega_1 < \hat \omega^*$, 
and $\alpha_3(\pm\gamma) \geq \beta_0 - 1$, as well as 
\[
\alpha_1(\pm\gamma) + \alpha_3(\mp\gamma) \leq \beta_0 + \frac{8\beta_1 - 2}{1+\gamma}.
\]
We see that all the above needed constraints are implied by the more severe constraints, including  \eqref{beta2d} for $\beta_0$, 
and the CFL condition \eqref{CFL2d}.

\bigskip

\section*{Acknowledgments}  Liu's research was partially supported by the National Science Foundation under Grant DMS1812666 and by NSF Grant RNMS (Ki-Net) 1107291. 


\end{document}